\documentclass[12pt]{amsart}

\usepackage{listings}  
\usepackage{amsfonts, amssymb, amsthm}

\usepackage{float} 

\usepackage{listings}  
\usepackage{arydshln}
  \usepackage{graphicx}

\usepackage[foot]{amsaddr}

\usepackage[dvipsnames]{xcolor}
\usepackage{hyperref}

\usepackage{tikz}
\usetikzlibrary{matrix,arrows,decorations.pathmorphing,automata, matrix, positioning, calc, shapes.multipart}
\newcommand\TikZ[1]{\begin{matrix}\begin{tikzpicture}#1\end{tikzpicture}\end{matrix}}
\usepackage[vcentermath]{youngtab}

\theoremstyle{definition}
\newtheorem{theorem}{Theorem}
\newtheorem{lemma}[theorem]{Lemma}

\newtheorem{corollary}[theorem]{Corollary}
\newtheorem{definition}[theorem]{Definition}
\newtheorem{remark}[theorem]{Remark}

\newtheorem{question}[theorem]{Question}
\newtheorem{example}[theorem]{Example}
\newtheorem{answer}[theorem]{Answer}

\DeclareMathOperator{\End}{End}

\DeclareMathOperator{\Ker}{Ker}

\DeclareMathOperator{\Res}{Res}
\DeclareMathOperator{\Spec}{Spec}

\usepackage{epigraph}

\usepackage{enumitem}

\usepackage{comment}

 \def\<{\langle}
  \def\>{\rangle}

\newcommand{\Hdeg}{\mathrm{H}^{\mathrm{deg}}} 
\newcommand{\sv}{\mathop{\mathrm{sV}\hspace{-2.5mm}\mathrm{V}}\nolimits} 
\newcommand{\vv}{\mathop{\mathrm{V}\hspace{-2.5mm}\mathrm{V}}\nolimits}
\newcommand{\sBr}{\mathsf{A}}
\newcommand{\TL}{\mathsf{TL}}
\newcommand{\Infl}{\mathop{\mathrm{Infl}}}

\newcommand{\K}{\mathbb{C}}
\newcommand{\Z}{\mathbb{Z}}
\newcommand{\N}{\mathbb{N}}
\newcommand{\diag}{\mathrm{diag}}

\newcommand{\CC}{\mathrm{C}}
\newcommand{\JJ}{\mathrm{J}}
\newcommand{\LL}{\mathrm{L}}
\newcommand{\MM}{\mathrm{M}}
\newcommand{\WW}{\mathrm{W}}

\newcommand{\Id}{\mathrm{Id}}

\newcommand{\cB}{\mathcal{B}}

\newcommand{\ct}{\mathrm{ct}}

\newcommand{\yminusone}{$-1$}
\newcommand{\yminustwo}{$-2$}
\newcommand{\yminusthree}{$-3$}
\newcommand{\yminusfour}{$-4$}
\newcommand{\yminusfive}{$-5$}

\allowdisplaybreaks

\begin{document}

\title[Calibrated representations of periplectic Brauer algebras]{Irreducible calibrated representations of periplectic Brauer algebras and hook representations of the symmetric group}

\author{Mee Seong Im}
\address[M.S.I.]{Department of Mathematical Sciences, United States Military Academy, West Point, NY 10996 USA}
\email{meeseongim@gmail.com}
\author{Emily Norton}
\address[E.N.]{Max Planck Institute for Mathematics, 53111 Bonn
Germany}
\email{enorton@mpim-bonn.mpg.de}



\begin{abstract}
  We construct an infinite tower of irreducible calibrated representations of periplectic Brauer algebras on which the cup-cap generators act by nonzero matrices. As representations of the symmetric group, these are exterior powers of the standard representation (i.e. hook representations). Our approach uses the recently-defined degenerate affine periplectic Brauer algebra, which plays a role similar to that of the degenerate affine Hecke algebra in representation theory of the symmetric group. We write formulas for the representing matrices in the basis of Jucys--Murphy eigenvectors and we completely describe the spectrum of these representations. The tower formed by these representations provides a new, non-semisimple categorification of Pascal's triangle. Along the way, we also prove some basic results about calibrated representations of the degenerate affine periplectic Brauer algebra.
\end{abstract}  
\maketitle


\section*{Introduction}

The periplectic Brauer algebra $\sBr_n$ is an example of a finite-dimensional algebra with a ``local system of generators," in the terminology of Okounkov and Vershik \cite{OkounkovVershik}. It is a diagram algebra of crossing lines and curving arcs including the usual generators $s_1,\dots,s_{n-1}$ and relations of the symmetric group $S_n$, echoed by an additional set of generators $e_1,\dots,e_{n-1}$ and relations which produce a signed Temperley-Lieb algebra $\TL^-_n$ with loops evaluated at $0$. Together, the signed Temperley--Lieb algebra and the symmetric group generate the periplectic Brauer algebra $\sBr_n$, according to additional rules governing the interaction between the two algebras $\K S_n$ and $\TL^-_n$ in $\sBr_n$. The result is a unique and exceptional algebra: if we impose any other choice of value $\delta$ for the loops, then the algebra cannot exist, and this sets it apart from the usual Brauer algebra $\mathsf{Br}_n(\delta)$ which can be defined for any loop-value $\delta$.

The representation theory of the periplectic Brauer algebra is like a partner dance between $S_n$ and $\TL_n^-$ with exaggerated binary gender roles. The symmetric group leads and the algebra of curves follows, embellishing and adding complexity. Cell modules are labeled not just by partitions of $n$, but by layer upon layer of partitions -- partitions of $n$, $n-2$, $n-4, \ldots$ and all of these cell modules have an irreducible quotient except for the cell module labeling the empty partition of $0$ when $n$ is even \cite{Coulembier}. The filtration of a cell module labeled by a partition $\mu$ contains an irreducible representation labeled by $\lambda$ in its Jordan--H\"{o}lder filtration exactly once if and only if the Young diagram of $\mu$ fits inside the Young diagram of $\lambda$ and their difference belongs to a certain set of skew Young diagrams, and this completely describes the composition series of the cell modules \cite[Theorem 1]{CoulembierEhrig1}. Many open questions remain, however, such as determining the branching graph of the irreducible representations.

A famous paper by Okounkov and Vershik \cite{OkounkovVershik} reinvents the representation theory of the symmetric group $S_n$ starting from the tower of algebras $\K\cong \K S_1\subset \K S_2\subset \K S_3\subset\dots\subset \K S_n\subset \K S_{n+1}\subset\dots$
and the commutative subalgebra of $\K S_n$ generated by the Jucys--Murphy elements $X_1,\dots,X_n$. It is interesting to consider some analogues of this theory for periplectic Brauer algebras. 
By adding a strand on the right again and again, there is a chain of inclusions $\K\cong \sBr_1\subset\sBr_2\subset\sBr_3\subset\dots\subset\sBr_n\subset\sBr_{n+1}\subset\dots$
 going on forever, making a tower of algebras. Each algebra $\sBr_n$ has a large commutative subalgebra generated by the periplectic Jucys--Murphy elements $Y_1,\dots,Y_n$ defined in \cite{Coulembier}. The Jucys--Murphy elements of $\sBr_n$ naturally generalize the Jucys--Murphy elements of $\K S_n$. Recall that the Jucys--Murphy elements of the symmetric group are defined recursively as:
 $$X_1:=0,\quad X_{j+1}=s_jX_js_j+s_j.$$
  The Jucys--Murphy elements of $\sBr_n$ are defined recursively as:
$$Y_1:=0,\quad Y_{j+1}=s_jY_js_j+s_j+e_j.$$
Just as $X_n$ commutes with $\K S_{n-1}$, $Y_n$ commutes with $\sBr_{n-1}$ under the embedding of the latter as above \cite{Coulembier}.

A major difference between $\K S_n$ and $\sBr_n$ is that $\sBr_n$ is not a semisimple algebra \cite{Moon},\cite{Coulembier}. In contrast to the situation for $\K S_n$, there is no reason an arbitrary irreducible representation of $\sBr_n$ should possess a basis of joint eigenvectors for the Jucys--Murphy elements $Y_j$. It is an interesting question to determine when this happens. 
 In this paper, we build a natural series of irreducible representations of $\sBr_n$ which have a basis on which $Y_j$ acts as a diagonal matrix for all $j=1,\dots,n$; we call representations with such a basis \textit{calibrated}. As an example of an irreducible calibrated $\sBr_n$-representation, one can always take $\LL(\lambda)$ for $\lambda$ a partition of $n$, because $\K S_n$-mod embeds in $\sBr_n$-mod via $S(\lambda)\rightarrow \LL(\lambda)$, but this is not an interesting example since the representation theory of the symmetric group is known. For a fixed $n$, our family of irreducible calibrated representations cuts transversely across the filtration layers of $\sBr_n$-mod. The calibrated representations we find are very natural: as $S_n$-representations they are the exterior powers of the standard representation of $S_n$, so are irreducible already as $S_n$-representations. The matrices by which the generators $e_i$ act are very similar to the matrices by which the $s_i$ act -- they are obtained by removing the diagonal entries and changing some signs. And there are simple recursive formulas for the eigenvalues of the Jucys--Murphy elements on these representations; the eigenvalues of $Y_{j+1}$ are obtained from the eigenvalues of $Y_j$ by adding or subtracting $1$.

In order to construct interesting calibrated representations of $\sBr_n$, we study the degenerate affine periplectic Brauer algebra $\sv_n$ defined in \cite{us+},\cite{ChenPeng}. This is the algebra generated by $\sBr_n$ and commuting formal variables $y_1,\dots,y_n$ satisfying the same relations with the generators $s_i$ and $e_i$ of $\sBr_n$ as the Jucys--Murphy elements $Y_j$. The definition and use of this algebra $\sv_n$ follow from the same philosophy that leads to the definition and use of the degenerate affine Hecke algebra $\Hdeg_n$ in the representation theory of $S_n$ as in the work of Okounkov and Vershik \cite{OkounkovVershik} and Kleshchev \cite{Kleshchev}. 

We now summarize the contents of the paper and our results. In Section \ref{algebras}, we introduce the algebras $\sBr_n$ and $\sv_n$ and discuss how they are related to each other as well as to more familiar algebras such as $\K S_n$ and the degenerate affine Hecke algebra, and what consequences this has for their internal structure and their representation theory. We explain how $\sv_n$ inherits a filtration by ``cup-cap ideals" from $\TL^-$, we describe the Jacobson radical of $\sv_2$ and put some upper and lower bounds on the Jacobson radical of $\sv_n$. In Section \ref{cali reps and JM elts} we study calibrated representations of $\sBr_n$ and $\sv_n$. We prove that the center of $\sv_n$ acts trivially on any calibrated $\sv_n$-representation, that is, viewed as a symmetric polynomial an element of $Z(\sv_n)$ acts by its constant term (Theorem \ref{Theta}, Corollary \ref{Theta cor}). Using some basic results about calibrated $\sv_2$-representations, which was the subject of a previous paper by the authors and other collaborators \cite{some of us}, we analyze the eigenvalues of $y_1,\dots,y_n$ on calibrated $\sv_n$-representations and show that the eigenvalues of $y_j$ are integers so long as the eigenvalues of $y_1$ are integers (Theorem \ref{integer evalues}); as a corollary, the eigenvalues of $Y_j$ on any calibrated $\sBr_n$-representation are always integers (Corollary \ref{evalues sBr ints}). Furthermore, we work out conditions under which we can write a simple formula for these eigenvalues (Lemma \ref{evalues closed form}); this formula is used to study the family of representations constructed in Section \ref{irreduciblecali}.

In Section \ref{irreduciblecali} we tackle the construction of a family of irreducible calibrated $\sBr_n$-representations endowed with a nonzero action of $\TL_n^-$ for any $n>2$. The starting point is Theorem \ref{reflection rep}. With the right choice of representing matrices for $s_i$, the $(n-1)$-dimensional standard representation $V_n$ of $S_n$ extends to a one-parameter family of calibrated representations $\CC_\alpha(V_n)$ of $\sv_n$ factoring through $\sBr_n$ when the parameter $\alpha$ is set equal to $0$. This gives us the first explicit example of an irreducible calibrated representation of $\sv_n$ or $\sBr_n$ for any $n>2$ which has a nonzero action of the $e_i$'s. Once we have the representation $\CC_\alpha(V_n)$ in hand, we compute the exterior powers $\Lambda^k V_n$ as $\K S_n$-representations with respect to the basis of $y$-eigenvectors, and then extend these to calibrated representations $\CC_\alpha(\Lambda^k V_n)$ of $\sv_n$ by writing formulas for the actions of $e_i$ and $y_j$ similar to those used for $\CC_\alpha(V_n)$. We prove in Theorem \ref{exterior powers} that this does indeed define a representation of $\sv_n$, and thus of $\sBr_n$ when $\alpha=0$. Since $\Lambda^k V_n$ is already irreducible as a $\K S_n$-representation, $\CC_\alpha(\Lambda^k V_n)$ is irreducible as an $\sBr_n$- or $\sv_n$-representation.

The results of Section \ref{irreduciblecali} considered for all $n$ give us an infinite series of irreducible calibrated representations of the tower of algebras $\sBr_1\subset\sBr_2\subset\sBr_3\subset\dots$ under the sequence of embeddings $\sBr_{n-1}\hookrightarrow\sBr_n$ adding a vertical strand on the right. A natural question is how these representations are related by restriction. In Section \ref{branching} \ref{res special}, we prove that our collection of irreducible calibrated representations $\{\CC_0(\Lambda^k V_n)\mid n\in\N,\;0\leq k\leq n-1\}$ is closed under taking composition factors of restriction from $\sBr_{n}$ to $\sBr_{m}$, $m<n$. Theorem \ref{res special} states that the restriction of $\CC_0(\Lambda^k V_n)$ from $\sBr_n$ to $\sBr_{n-1}$ is always indecomposable with two irreducible composition factors if $1\leq k\leq n-2$: the head $\CC_0(\Lambda^{k-1}V_{n-1})$ and the socle $\CC_0(\Lambda^k V_{n-1})$. As for the representations $\CC_0(\Lambda^0 V_n)$ and $\CC_0(\Lambda^{n-1} V_n)$, they are just the trivial and the sign representation of $S_n$, respectively, with $e_i$ acting by $0$.

The Bratteli diagram (i.e. branching graph) of  $\{\CC_0(\Lambda^k V_n)\mid n\in\N,\;0\leq k\leq n-1\}$ thus yields a new categorification of Pascal's triangle
with a non-semisimple rule for the arrows in the interior of the triangle. This is not the first time Pascal's triangle has been categorified by a branching rule in a module category: other notable instances of Pascal categorification include the category of modules over planar rook algebras (a semisimple categorification) \cite{PascalTriangle} and the branching rule for the standard modules over the blob algebra \cite{TheBlob} (furthermore, the irreducible quotients of these blob modules have bases given by paths in Pascal's triangle that avoid certain hyperplanes, see \cite{BNS} and take the case $\ell=2$ and $h=1$).
Our Bratteli diagram also gives a combinatorial rule for computing the eigenvalues of the Jucys--Murphy elements on any irreducible calibrated $\sBr_n$-representation of the form $\CC_0(\Lambda^k V_n)$. We substitute one-row and one-column partitions for the representations $\CC_0(\Lambda^k V_n)$ forming the vertices of the Bratteli diagram, and then consider all directed paths from the source vertex to a vertex $\lambda$. The content sequence determined by the unique removable box of each partition $\mu$ along such a path gives an element of the spectrum of $\CC_0(\Lambda^k V_n)$, where $\lambda$ is determined from $k$ and $n$ by an easy formula. See Section \ref{branching}.

\section{The finite and degenerate affine periplectic Brauer algebras }\label{algebras}
\subsection{Definitions of the algebras} The degenerate affine Hecke algebra, also called the graded Hecke algebra, has an easy presentation by generators and relations.
\begin{definition}\cite{Drinfeld},\cite{Lusztig} The degenerate affine Hecke algebra $\Hdeg_n$ is the $\K$-algebra generated by $s_i$ for $i=1,\dots,n-1$ and by $y_j$ for $j=1,\dots,n$ with relations:
\begin{align*}
&s_i^2=1\hbox{ for }1\leq i\leq n-1,\quad s_is_{i+1}s_i=s_{i+1}s_is_{i+1}\hbox{ for }1\leq i\leq n-2,\\
&s_is_j=s_js_i\hbox{ if }|j-i|>1, \quad y_iy_j=y_jy_i \hbox{ for }1\leq i, j\leq n,\\&y_{j+1}=s_jy_js_j+s_j \hbox{ for }1\leq j\leq n-1.
\end{align*}
\end{definition}
\noindent We see the group algebra $\K S_n$ of $S_n$ defined in the first three relations. The last relation echoes the recursive definition of the Jucys--Murphy elements of $\K S_n$, but since $y_1$ is a free variable, in $\Hdeg_n$ these have become abstract entities $y_j$ which generate a polynomial algebra. The idea behind the degenerate affine periplectic Brauer algebra is similar and has a Brauer algebra precursor in \cite{Nazarov}.

\begin{definition}\label{def svn} \cite[Def. 3.1]{ChenPeng}, \cite[Def. 39]{us+}
The degenerate affine periplectic Brauer algebra $\sv_n$ is the $\K$-algebra generated by $s_i$ and $e_i$ for $i=1,\dots,n-1$ and by $y_j$ for $j=1,\dots,n$ subject to the relations:
\begin{align*}
(\vv_1)\;&s_i^2=1\hbox{ for }1\leq i<n,\\
(\vv_2)\;&\mathrm{(i)}\;s_ie_j=e_js_i\hbox{ if }|i-j|>1, \quad\;\;\; \mathrm{(ii)}\; e_ie_j=e_je_i\hbox{ if }|i-j|>1,\\
&\mathrm{(iii)}\;e_iy_j=y_je_i\hbox{ if }j\neq i,i+1,\quad\mathrm{(iv)}\;y_iy_j=y_jy_i\hbox{ for }1\leq i,j\leq n,\\
(\vv_3)\;&\mathrm{(i)}\;s_is_j=s_js_i\hbox{ if }|i-j|>1,\quad \mathrm{(ii)}\;s_is_{i+1}s_i=s_{i+1}s_is_{i+1}\hbox{ for }1\leq i\leq n-2,\\
&\mathrm{(iii)}\;s_iy_j=y_js_i\hbox{ if }j\neq i,i+1,\\
(\vv_4)\;&\mathrm{(i)}\;e_{i+1}e_ie_{i+1}=-e_{i+1},\quad \mathrm{(ii)}\;e_ie_{i+1}e_i=-e_i\hbox{ for }1\leq i\leq n-2,\\
(\vv_5)\;&\mathrm{(i)}\;e_is_i=e_i\hbox{ and }s_ie_i=-e_i\hbox{ for }1\leq i\leq n-1,\quad\mathrm{(ii)}\;s_ie_{i+1}e_i=s_{i+1}e_i, \\
&\mathrm{(iii)}\;s_{i+1}e_ie_{i+1}=-s_ie_{i+1},\quad \mathrm{(iv)}\;e_{i+1}e_is_{i+1}=e_{i+1}s_i,\quad\mathrm{(v)}\;e_ie_{i+1}s_i=-e_is_{i+1}\\&\hbox{ for }1\leq i\leq n-2,\\
(\vv_6)\;&e_i^2=0\hbox{ for }1\leq i\leq n-1,\\
(\vv_7)\;&\mathrm{(i)}\;s_iy_i-y_{i+1}s_i=-e_i-1,\quad\mathrm{(ii)}\;y_is_i-s_iy_{i+1}=e_i-1\hbox{ for }1\leq i\leq n-1,\\
(\vv_8)\;&\mathrm{(i)}\;e_i(y_i-y_{i+1})=-e_i,\quad\mathrm{(ii)}\;(y_i-y_{i+1})e_i=e_i \hbox{ for }1\leq i\leq n-1.\\
\end{align*}
\end{definition}

\begin{remark}\label{better reln} Relation ($\vv_7$) can be conveniently packaged as the single relation $$y_{j+1}=s_jy_js_j+s_j+e_j$$ emphasizing the origin of the $y_j$'s as Jucys--Murphy elements -- multiply each side of the equation by $s_i$ to recover ($\vv_7$). The algebra $\sv_n$ is generated by $e_i$ and $s_i$, $1\leq i\leq n-1$, together with $y_1$.
\end{remark}

\begin{remark}\label{less relns}
Our sign conventions follow \cite[Def. 39]{us+} and differ in some relations from those in \cite[Def. 3.1]{ChenPeng},\cite{Moon},\cite{Coulembier}. We have omitted the relation $e_1y_1^ke_1=0$ which was proved in (\cite{us+},\cite{ChenPeng}) to follow from the other relations. The list of relations above is still not minimal \cite[Rem. 40]{us+}. For example, $(\vv_5)$(ii) follows by multiplying $(\vv_5)$(iii) on the right by $e_i$ then simplifying using $(\vv_4)$(ii). Similarly, $(\vv_5)$(iv) follows from $(\vv_5)$(v) and $(\vv_4)$(i).
\end{remark}

The algebra $\sv_n$ has several important subalgebras. First, there is the polynomial subalgebra $\K [y_1,\dots,y_n]$. Second, there is the periplectic Brauer algebra.
\begin{definition}
The periplectic Brauer algebra $\sBr_n$ is the subalgebra of $\sv_n$ generated by $e_i$ and $s_i$ for $1\leq i\leq n-1$.
 \end{definition}
\noindent  This algebra was first introduced by Moon to study Schur--Weyl duality for the periplectic Lie superalgebra $\mathfrak{p}(n)$ \cite{Moon} and has recently been the focus of much attention \cite{KujawaTharp},\cite{Coulembier},\cite{CoulembierEhrig1},\cite{CoulembierEhrig2}. 
The traditional way to work with such algebras is as diagram algebras: $s_i$ is a \textit{crossing},  $e_i$ is a \textit{cup-cap}, and the elements of $\sBr_n$ are linear combinations of \textit{Brauer diagrams}. A Brauer diagram is a  pairing of $2n$ points, $n$ of them equally spaced on a horizontal line of height $0$ and $n$ of them equally spaced directly above on a horizontal line of height $1$. If two points are paired, they are drawn with a line or \textit{strand} connecting them. We do not draw the $2n$ points in our diagrams but only the strands. The strands which connect a point at the bottom of the diagram with a point at the top of the diagram are called \textit{through strands}. The strands connecting points on the top horizontal line are called \textit{cups}; on the bottom horizontal line, \textit{caps}.
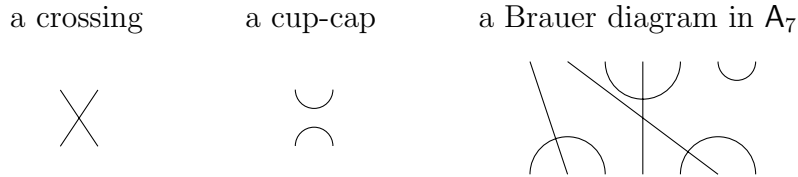
\begin{example}
\begin{align*}&\hbox{a crossing}\qquad\quad\;\hbox{a cup-cap}\qquad\quad\;\hbox{a Brauer diagram in }\sBr_7\\
&\quad\;\TikZ{[scale=.5] \draw
(0,0)node{} to (1,1.5)node{}
(1,0)node{} to (0,1.5)node{}
;}\qquad\qquad\qquad
\TikZ{[scale=.5] \draw
(0,0) arc(180:0:0.5)
(0,1.5) arc(-180:0:.5)
;}\qquad\qquad\qquad
\TikZ{[scale=.5]\draw
(0,0) arc(180:0:1)
(1,0)node{} to (0,3)node{}
(3,0)node{} to (3,3)node{}
(1,3)node{} to (5,0)node{}
(4,0) arc(180:0:1)
(2,3) arc(-180:0:1)
(5,3) arc(-180:0:.5)
;}
\end{align*}
\end{example}
\noindent Multiplication $ab$ of elements $a,b\in\sBr_n$ corresponds to stacking the diagram of $a$ on top of the diagram of $b$. Any closed loop occurring in a diagram makes the whole diagram $0$:
$$ \TikZ{[scale=.5] \draw
(0,0) arc(360:0:0.5)
;}=0.
$$

  The elements $s_1,\dots,s_{n-1}$ generate the subalgebra $\K S_n\subset \sBr_n\subset \sv_n$, while the elements $e_1,\dots, e_{n-1}$ together with $1$ generate a copy of an algebra we denote by $\TL^-_n\subset \sBr_n\subset \sv_n$, which can be thought of as a signed or super version of a Temperley--Lieb algebra, again with any closed loop evaluated at $0$. Note that the subalgebra generated by $e_1,\dots,e_{n-1}$ is not unital; we throw in $1$ to make $\TL^-_n$ a unital subalgebra. Together the algebras $\TL^-_n$ and $\K[y_1,\dots,y_n]$ also generate a subalgebra of $\sv_n$. However, $\K S_n$ and $\K[y_1,\dots,y_n]$ together generate all of $\sv_n$ as every $e_i=y_{i+1}-s_iy_is_i-s_i$, see Remark \ref{better reln}.

\begin{align*}
&\hbox{an element of }S_5\qquad\qquad\qquad\hbox{an element of }\TL_5^-\\
&
\TikZ{[scale=.5]\draw
(0,0)node{} to (3,3)node{}
(1,0)node{} to (4,3)node{}
(2,0)node{} to (1,3)node{}
(3,0)node{} to (2,3)node{}
(4,0)node{} to (0,3)node{}
;}
\qquad\qquad\qquad\qquad\TikZ{[scale=.5]\draw
(0,0) arc(180:0:.5)
(2,0) arc(180:0:.5)
(1,3) arc(-180:0:1.5)
(2,3) arc(-180:0:.5);
\draw[] plot [smooth,tension=.7] coordinates{(0,3)(.5,2)(1.5,1.5)(3.5,1)(4,0)};
;}
\end{align*}  

An intuitive way to think about $\sv_n$ is as the diagram algebra generated by the diagram algebra $\sBr_n$ together with $\K[y_1,\dots,y_n]$ where $y_i$ is a bead in position $i$ at the top or bottom of the diagram, subject to some local relations for moving a bead across a cup or cap, or through a crossing:
\begin{align*}
&\TikZ{[scale=.5] \draw
(0,0)node{} to (1,1.5)node{}
(1,0)node{} to (0,1.5)node{}
(0,0)node[fill,circle,inner sep=2pt]{}
;} = 
\TikZ{[scale=.5] \draw
(0,0)node{} to (1,1.5)node{}
(1,0)node{} to (0,1.5)node{}
(1,1.5)node[fill,circle,inner sep=2pt]{}
;}
-
\TikZ{[scale=.5] \draw
(0,0)node{} to (0,1.5)node{}
(1,0)node{} to (1,1.5)node{}
;}
-
\TikZ{[scale=.5] \draw
(0,0) arc (180:0:.5)
(0,1.5) arc (-180:0:.5)
;}
\qquad\qquad
\TikZ{[scale=.5] \draw
(0,0)node{} to (1,1.5)node{}
(1,0)node{} to (0,1.5)node{}
(1,0)node[fill,circle,inner sep=2pt]{}
;} = 
\TikZ{[scale=.5] \draw
(0,0)node{} to (1,1.5)node{}
(1,0)node{} to (0,1.5)node{}
(0,1.5)node[fill,circle,inner sep=2pt]{}
;}
+
\TikZ{[scale=.5] \draw
(0,0)node{} to (0,1.5)node{}
(1,0)node{} to (1,1.5)node{}
;}
-
\TikZ{[scale=.5] \draw
(0,0) arc (180:0:.5)
(0,1.5) arc (-180:0:.5)
;}
\\
\\
&
\qquad\qquad
\TikZ{[scale=.5] \draw
(0,0) arc (180:0:.5)
(0,1.5) arc (-180:0:.5)
(1,0)node[fill,circle,inner sep=2pt]{}
;} = 
\TikZ{[scale=.5] \draw
(0,0) arc (180:0:.5)
(0,1.5) arc (-180:0:.5)
(0,0)node[fill,circle,inner sep=2pt]{}
;} 
+
\TikZ{[scale=.5] \draw
(0,0) arc (180:0:.5)
(0,1.5) arc (-180:0:.5)
;} 
\qquad\qquad
\TikZ{[scale=.5] \draw
(0,0) arc (180:0:.5)
(0,1.5) arc (-180:0:.5)
(0,1.5)node[fill,circle,inner sep=2pt]{}
;} = 
\TikZ{[scale=.5] \draw
(0,0) arc (180:0:.5)
(0,1.5) arc (-180:0:.5)
(1,1.5)node[fill,circle,inner sep=2pt]{}
;} 
+
\TikZ{[scale=.5] \draw
(0,0) arc (180:0:.5)
(0,1.5) arc (-180:0:.5)
;} 
\end{align*}
and such that beads that are ``far away" (somewhere off to the left or right) commute with a crossing or a cup-cap, and the beads commute with each other.
It is not at all obvious that such an algebra has the basis one would like: a basis of Brauer diagrams with some number of beads (corresponding to $y_i$'s) on each string, pushed to an agreed-upon end of the string in a neat and orderly fashion. Such a beaded Brauer diagram is called a \textit{normal diagram}. Rather than belabor the precise definition, we draw a picture of a normal diagram in $\sv_7$:

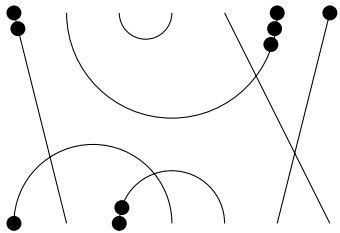
$$\TikZ{[scale=.7] \draw
(1,4) arc (-180:0:2)
(2,4) arc (-180:0:0.5)
(0,0) arc (180:0:1.5)
(2,0) arc (180:0:1)
(5,0)node{} to (6,4)node{}
(6,0)node{} to (4,4)node{}
(1,0)node{} to (0,4)node{}
(5,4)node[fill,circle,inner sep=2pt]{}
(4.95,3.7)node[fill,circle,inner sep=2pt]{}
(4.88,3.4)node[fill,circle,inner sep=2pt]{}
(0,0)node[fill,circle,inner sep=2pt]{}
(2,0)node[fill,circle,inner sep=2pt]{}
(2.05,0.3)node[fill,circle,inner sep=2pt]{}
(0,4)node[fill,circle,inner sep=2pt]{}
(0.075,3.7)node[fill,circle,inner sep=2pt]{}
(6,4)node[fill,circle,inner sep=2pt]{}
;}
$$
An important first result about the algebra $\sv_n$ is therefore:
\begin{theorem}\label{PBW}\cite{ChenPeng},\cite{us+}
The algebra $\sv_n$ has a $\K$-basis consisting of all normal diagrams.
\end{theorem}
\noindent This can be thought of as an analogue of classical PBW theorems for universal enveloping algebras and flat deformations of skew group rings \footnote{The proof in \cite{ChenPeng} is a signed version of the proof given by Nazarov for the degenerate affine Brauer algebra \cite{Nazarov}. The proof in \cite{us+} proceeds by constructing an explicit faithful representation of $\sv_n$ using the intrinsic connection of $\sv_n$ with periplectic Lie superalgebras \cite[Theorem 2]{us+}.}.

\subsection{Some relations in $\sv_n$}

In addition to the formulas written in \cite[Lemmas 4-9]{us+} there are nice formulas for disentangling a bunch of beads stuck between a crossing and a cup or cap. The proofs of the formulas follow from easy induction arguments and the formulas in \cite[Lemmas 4-9]{us+}.

\begin{lemma}\label{disentangling} Let $a,b\in\mathbb{N}$. The following equalities hold in $\sv_n$:
\begin{align*}
(i) \; &s_iy_{i+1}^be_i=-y_{i+1}^be_i, \qquad (ii) \; s_iy_i^ae_i=-y_i^ae_i, \qquad\quad (iii) \; e_iy_{i+1}^b s_i = e_iy_{i+1}^b, \\
&\TikZ{[scale=.5]
\node[right] at (1,1){b};
\draw
(0,1)node{} to (1,2)node{}
(1,1)node{} to (0,2)node{}
(0,1) arc(-180:0:0.5)
(0,-.5) arc(180:0:0.5)
(1,1)node[fill,circle,inner sep=2pt]{}
;}
=
-\TikZ{[scale=.5]
\node[right] at (1,1){b};
\draw
(0,1)arc(-180:0:0.5)
(0,-.5) arc(180:0:0.5)
(1,1)node[fill,circle,inner sep=2pt]{}
;}
\qquad\quad
\TikZ{[scale=.5]
\node[left] at (0,1){a};
\draw
(1,1)node{} to (0,2)node{}
(0,1)node{} to (1,2)node{}
(0,1) arc(-180:0:0.5)
(0,-.5) arc(180:0:0.5)
(0,1)node[fill,circle,inner sep=2pt]{}
;}
=
-\TikZ{[scale=.5]
\node[left] at (0,1){a};
\draw
(0,1)arc(-180:0:0.5)
(0,-.5) arc(180:0:0.5)
(0,1)node[fill,circle,inner sep=2pt]{}
;}
\qquad\qquad
\TikZ{[scale=.5] 
\node[right] at (1,0){b};
\draw
(0,-1)node{} to (1,0)node{}
(1,-1)node{} to (0,0)node{}
(0,0) arc(180:0:0.5)
(0,1.5) arc(-180:0:0.5)
(1,0)node[fill,circle,inner sep=2pt]{}
;}
=
\TikZ{[scale=.5]
\node[right] at (1,0){b};
\draw
(0,0)arc(180:0:0.5)
(0,1.5) arc(-180:0:0.5)
(1,0)node[fill,circle,inner sep=2pt]{}
;}
\\
\\
(iv) \; &e_iy_i^{a}s_i=e_iy_i^a, \qquad (v)\; e_iy_i^ay_{i+1}^b s_i= \sum_{k=0}^b {b\choose k}e_i y_i^{a+k}, \quad (vi)\;  s_iy_i^ay_{i+1}^be_i=-\sum_{k=0}^a{a\choose k}y_{i+1}^{b+k}e_i. \\
& \TikZ{[scale=.5] 
\node[left] at (0,0){a};
\draw
(0,-1)node{} to (1,0)node{}
(1,-1)node{} to (0,0)node{}
(0,0) arc(180:0:0.5)
(0,1.5) arc(-180:0:0.5)
(0,0)node[fill,circle,inner sep=2pt]{}
;}
=
\TikZ{[scale=.5]
\node[left] at (0,0){a};
\draw
(0,0)arc(180:0:0.5)
(0,1.5) arc(-180:0:0.5)
(0,0)node[fill,circle,inner sep=2pt]{}
;}
\qquad
\TikZ{[scale=.5] 
\node[left] at (0,-1){a};
\node[right] at (1,-1){b};
\draw
(0,-2)node{} to (1,-1)node{}
(1,-2)node{} to (0,-1)node{}
(0,-1) arc(180:0:0.5)
(0,.5) arc(-180:0:0.5)
(1,-1)node[fill,circle,inner sep=2pt]{}
(0,-1)node[fill,circle,inner sep=2pt]{}
;}
=
\sum_{k=0}^b  \binom{b}{k}
\TikZ{[scale=.5]
\node[left] at (0,0){a+k};
\draw
(0,0)arc(180:0:0.5)
(0,1.5) arc(-180:0:0.5)
(0,0)node[fill,circle,inner sep=2pt]{}
;} 
\qquad
\TikZ{[scale=.5]
\node[left] at (0,1){a};
\node[right] at (1,1){b};
\draw
(1,1)node{} to (0,2)node{}
(0,1)node{} to (1,2)node{}
(0,1) arc(-180:0:0.5)
(0,-.5) arc(180:0:0.5)
(0,1)node[fill,circle,inner sep=2pt]{}
(1,1)node[fill,circle,inner sep=2pt]{}
;}
=
-\sum_{k=0}^a  \binom{a}{k}
\TikZ{[scale=.5]
\node[right] at (1,0){b+k};
\draw
(0,0)arc(-180:0:0.5)
(0,-1.5) arc(180:0:0.5)
(1,0)node[fill,circle,inner sep=2pt]{}
;} 
\end{align*}
\end{lemma}

\subsection{Basics from representation theory}Write $\sv_n$-mod for the category of finite-dimensional $\sv_n$-representations and $\sBr_n$-mod for the category of finite-dimensional $\sBr_n$-representations.
\textit{We will only consider finite-dimensional representations in this paper and may often just write ``representation" when we mean finite-dimensional representation.} 

\noindent 
Practically speaking, we will think of the elements of $V$ as column vectors in $\K^{\dim V}$ with respect to a fixed basis, and if $\rho:\sv_n\rightarrow\End(V)$ is a representation then we will think of the linear transformations $\rho(h)$, $h\in\sv_n$, as matrices with respect to that basis. We will abuse notation and define representations by writing the generators directly as matrices, dropping the notation $\rho$.


The categories $\sv_n$-mod and $\sBr_n$-mod are Krull-Schmidt, that is, any finite-dimensional $\sv_n$- or $\sBr_n$-representation has a unique decomposition as a direct sum of indecomposable representations up to isomorphism and permuting the factors. However, $\sv_n$-mod and $\sBr_n$-mod are not semisimple, that is, not every finite-dimensional $\sv_n$- or $\sBr_n$-representation can be decomposed into a direct sum of irreducible representations. In order to determine when a representation $V$ of an algebra $A$ is indecomposable, we will use the criterion that $V$ is indecomposable if and only if $\End_A(V)$ is a local ring. In particular, if $\End_A(V)\cong \K$ then certainly $V$ is indecomposable (but not necessarily irreducible).


\subsection{Quotient maps and inflated representations}

It follows immediately from Definition \ref{def svn} and Theorem \ref{PBW} that the degenerate affine Hecke algebra is a quotient of the degenerate affine periplectic Brauer algebra:
\begin{lemma}\label{Hdeg quotient}
There is a surjective homomorphism of algebras $$\Phi_n:\sv_n\twoheadrightarrow\Hdeg_n$$ given by quotienting $\sv_n$ by the two-sided ideal $(e_1,e_2,\dots,e_{n-1})$.
\end{lemma} 
\noindent This allows us to construct a pile of representations of $\sv_n$ for free by ``inflation":
\begin{definition}
Let $\overline{\MM}$ be an $\Hdeg_n$-representation. Define an $\sv_n$ representation $\MM:=\Infl\overline{\MM}$ by precomposing the $\Hdeg$-action on $\overline{\MM}$ with the map $\Phi$. That is, $\MM$ has the same underlying $\K$-vector space as $\overline{\MM}$ with $s_i$ and $y_j$ acting the same on $\MM$ as on $\overline{\MM}$, and $e_i$ acting by $0$.
\end{definition}
\noindent In particular, any irreducible representation of $\Hdeg_n$ gives rise to an irreducible representation of $\sv_n$ by inflation.

In addition to the surjection $\Phi_n$, $\sv_n$ also comes equipped with a surjection to $\sBr_n$.\begin{definition}\label{def JM} The $j$'th Jucys--Murphy element $Y_j$ of $\sBr_n$ is defined inductively by $Y_1:=0$, $Y_{j+1}:=s_jY_js_j+s_j+e_j$ for $1\leq j\leq n-1$.
\end{definition}
\begin{lemma}\label{JM eval}\cite{us+} There is a surjective algebra homomorphism
$$\Pi_n: \sv_n\twoheadrightarrow \sBr_n$$
given by $\Pi_n(e_i)=e_i,$ $\Pi_n(s_i)=s_i$, and $\Pi_n(y_j)=Y_j$, the $j$'th Jucys--Murphy element, for $1\leq i\leq n-1$ and $1\leq j\leq n$.
\end{lemma}
\noindent The kernel of $\Pi_n$ is the two-sided ideal $(y_1)$. Any irreducible representation $\LL$ of $\sBr_n$ also gives rise by inflation to an irreducible representation $\LL=\Infl\LL$ of $\sv_n$ by declaring $y_1$ to act by $0$, thus extending the action from $\sBr_n$ to all of $\sv_n$ by making $y_j$ act by the $j$'th Jucys--Murphy element $Y_j=\Pi_n(y_j)$.

The algebra $\sBr_n$ not only contains $\K S_n$ as the subalgebra generated by $s_1,\dots,s_{n-1}$ but also has $\K S_n$ as a quotient by the two-sided ideal $(e_1)=(e_1,\dots,e_{n-1})$. The irreducible representations of $\K S_n$ are labeled by partitions of $n$. Using the same method of inflating representations along the quotient map then yields, for every partition $\lambda$ of $n$, an irreducible $\sBr_n$-representation $\LL(\lambda)=\Infl S(\lambda)$ where $S(\lambda)$ is the irreducible $\K S_n$-representation labeled by $\lambda$. By construction, every $e_i$ acts on $\LL(\lambda)$ by $0$. The module category $\sBr_n$-mod thus contains $\K S_n$-mod as the full subcategory generated by all the $\LL(\lambda)$ for $\lambda$ a partition of $n$.

\subsection{Filtration by cup-cap ideals}\label{cup-cap ideals} There are filtrations of $\sBr_n$ and $\sv_n$ by two-sided ideals consisting of diagrams with ``at least $k$ cup-caps," equivalently, ``at most $n-2k$ through strands." The cup-cap filtration organizes the parametrization of irreducible $\sBr_n$-representations by partitions of $n$, partitions of $n-2$, partitions of $n-4, \ldots$ (except that the unique partition $\emptyset$ of $0$ does not label an irreducible representation) \cite{Coulembier}. The irreducibles labeled by partitions of $n$ have $e_i$ acting by $0$ (as we just saw), the ones labeled by partitions of $n-2$ have $e_i$ acting nontrivially but $e_{i_1}e_{i_2}$ acting by $0$ for $|i_1-i_2|>1$, and so on. We expect the cup-cap filtration to play a similar role in the description of irreducible $\sv_n$-representations.

If we do not say otherwise, ``ideal" will always mean two-sided ideal.

\begin{lemma}\label{cupcap ideal 1} Consider $e_i$ for an arbitrary $i\in\{2,\dots,n-1\}$.  The following statements are true:
\begin{enumerate}
\item The ideal $(e_1)$ is equal to the ideal $(e_i)$ for any $1<i\leq n-1$, and moreover, this is true in the subalgebra $\TL^-_n$ generated by $e_1,\dots,e_{n-1}$, and $1$;
\item Let $\MM\in\sv_n-$mod, then $e_i$ acts by $0$ on $\MM$ if and only if $e_1$ acts by $0$ on $\MM$;
\item\label{symm subrep} Let $\MM\in\sBr_n-$mod and suppose $\LL\subset\MM$ is a subspace preserved by $\K S_n$. Then either no $e_i$ preserves $\LL$, or $\LL$ is an $\sBr_n$-subrepresentation of $\MM$.
\end{enumerate}
\end{lemma}
\begin{proof}
\begin{enumerate}
\item 
Let $1<i\leq n-1$. We have 
$$e_i=-e_ie_{i-1}e_i=e_i\left(e_{i-1}e_{i-2}e_{i-1}\right)e_i=\ldots=(-1)^{i-1}e_ie_{i-1}e_{i-2}\cdots e_2e_1e_2\cdots e_{i-2}e_{i-1}e_i.$$
and
$$e_1=-e_1e_2e_1=e_1e_2e_3e_2e_1=\ldots=(-1)^{i-1}e_1e_2\cdots e_{i-2}e_{i-1}e_ie_{i-1}e_{i-2}\cdots e_2e_1.$$
\item This follows from the first part: if the generator of a principal ideal annihilates a module then clearly so does every element in that ideal.
\item 
 Using relations $(\sv_5)$ in Definition \ref{def svn} shows that for all $k=2,\dots,n-1$:
$$s_{k-1}e_ks_{k-1}=s_{k-1}e_ke_{k-1}s_k=-s_ke_{k-1}e_ke_{k-1}s_k=s_ke_{k-1}s_k,$$
 so for any $i\neq j$, $e_i=we_jw^{-1}$ for some $w\in S_n$. It follows that if some $e_i$ preserves an $S_n$-subrepresentation $\LL$ then all of $\sBr_n$ does.
\end{enumerate}
\end{proof}
\noindent Consequently, the ideal $\Ker\Phi_n$ is a principal ideal and any $e_i$, $i\in\{1,\dots,n-1\}$, can be taken as its generator. That is to say, we can obtain $\Hdeg_n$ from $\sv_n$ by just setting $e_i=0$ for some $i$.

In fact, the ideal $(e_1)$ is not only a principal ideal but has an interesting filtration by principal ideals. Set $I_0=\sv_n$, $I_1=(e_1)$, $I_2=(e_1e_3)$, $I_3=(e_1e_3e_5),\dots,$ $I_p=(e_1e_3e_5\dots e_{2p-1})$, $I_{p+1}=\{0\}$, where $p=\frac{n}{2}$ if $n$ is even and $p=\frac{n-1}{2}$ if $n$ is odd.
\begin{lemma}
\begin{enumerate}\item The ideal $I_r$ consists of all elements of $\sv_n$ whose diagrams contain at most $n-2r$ through strands. It can be generated by any element of the form $e_{i_1}e_{i_2}\cdots e_{i_r}$ such that $|i_j-a_k|>1$ for all $j,k=1,\dots,r$.
\item There is a filtration of $\sv_n$ given by $$0=I_{p+1}\subsetneq I_p=(e_1e_3\cdots e_{2p-1})\subsetneq I_{p-1}=(e_1e_3\cdots e_{2p-3})\subsetneq \ldots\subsetneq I_1=(e_1)\subsetneq I_0=\sv_n.$$
\end{enumerate}
\end{lemma}
\begin{proof}
The statement is obvious for $\sBr_n$. For $\sv_n$, the relations $e_i(y_i-y_{i+1})=-e_i$ and $(y_i-y_{i+1})e_i=e_i$ of Definition \ref{def svn} ($\vv_8$) preserve the number of cup-caps in a diagram, while the relation $y_{i+1}=s_iy_is_i+s_i+e_i$ ($\vv_7)$ preserves the number of cup-caps in the leading term and has a lower order term with one more cup-cap. Far away things commuting (Definition \ref{def svn} ($\vv_2$),($\vv_3$)) obviously preserves the number of cup-caps. Using the relations $e_i=-e_ie_{i\pm 1}e_i$ it is clear as in Lemma \ref{cupcap ideal 1} that $e_1e_3\cdots e_{2k-1}$ generates the ideal of diagrams with at most $k$ cup-caps, and that any other collection of $k$ independent cup-caps will also do the job.
\end{proof}


\subsection{The Jacobson radical}\label{J rad} The determination of the Jacobson radical of $\sBr_n$ and $\sv_n$ is an open problem. In this section we establish some obvious lower and upper bounds on the Jacobson radical of $\sv_n$ in terms of the ideals introduced in Section \ref{cup-cap ideals} and we determine $\JJ(\sv_2)$. 

Recall that if $A$ is an associative algebra, the Jacobson radical $\JJ(A)$ is defined to be the intersection of all maximal left ideals of $A$. Well-known facts about it include \cite{EncyclMath}:
\begin{lemma}\label{J} 
\begin{enumerate}
\item\label{J inv} Let $z\in A$, then $z\in\JJ(A)$ if and only if $1-az$ has a left inverse for all $a\in A$;
\item\label{J 2sided} The Jacobson radical $\JJ(A)$ is a two-sided ideal of $A$;
\item \label{J mod} $\JJ\left(A/\JJ(A)\right)=0$;
\item\label{J ss} $A$ is semisimple if and only if $\JJ(A)=0$ and $A$ is artinian;
\item\label{J idem} If $e^2=e$ then $e\notin\JJ(A)$;
\item\label{J surj} If $\phi: A\twoheadrightarrow B$ is a surjective homomorphism of algebras then $\phi(\JJ(A))\subseteq \JJ(B)$;
\item\label{J anni} $\JJ(A)$ acts by $0$ on every irreducible $A$-representation.
\end{enumerate}
\end{lemma}

Let $\JJ(\sBr_n)$ be the Jacobson radical of $\sBr_n$ and let $\JJ(\sv_n)$ be the Jacobson radical of $\sv_n$. Lemma \ref{JM eval} provides a surjective algebra homomorphism $\Pi_n:\sv_n\twoheadrightarrow\sBr_n$, so by Lemma \ref{J}(\ref{J surj}) we know that $\Pi_n(\JJ(\sv_n))\subseteq \JJ(\sBr_n)$.  The algebra $\sBr_n$ is finite-dimensional and non-semisimple \cite{Moon}, so by Lemma \ref{J}(\ref{J ss}) it holds that $\JJ(\sBr_n)\neq 0$ \cite{Coulembier}. For very small $n$ it is possible to calculate $\JJ(\sBr_n)$ by hand: when $n=2$, $\sBr_2$ is only $3$-dimensional with $\K$-basis $\{1,s,e\}$, and $\JJ(\sBr_2)=\K e$. When $n=3$, $\sBr_3$ is $15$-dimensional. Moon calculated $\JJ(\sBr_3)$ and found that it is $5$-dimensional; consequently, the maximal semisimple quotient of $\sBr_3$ has dimension $\dim_\K\sBr_3/\JJ(\sBr_3)=10.$
For any $n>3$, Coulembier defined a central element $\Theta$ given by a polynomial in the Jucys--Murphy elements $Y_2,\dots,Y_n$ and he proved that  $\Theta\in \JJ(\sBr_n)$ \cite[Equation (6.3)]{Coulembier}:
$$\Theta=\prod_{2\leq i<j\leq n}(1-(Y_i-Y_j)^2)=(1-(Y_2-Y_3)^2)(1-(Y_2-Y_4)^2)\cdots (1-(Y_{n-1}-Y_n)^2).$$
 However, it is unknown if $\Theta$ is nonzero for general $n$, see the remark preceding \cite[Proposition 6.4.3]{Coulembier}.

We could not find the next lemma in the literature, so we include a proof.
\begin{lemma}\label{dunno}
Let $\JJ(\Hdeg_n)$ denote the Jacobson radical of $\Hdeg_n$. Then $\JJ(\Hdeg_n)=0$.
\end{lemma}
\begin{proof}
Let $h\in \JJ(\Hdeg_n)$ and using the relation $y_{i+1}=s_iy_is_i+s_i$, write $h$ as a polynomial in $y_1$ with coefficients in $\K[S_n]$. As a polynomial in $y_1$, $h$ has the same degree as it does as a polynomial in $y_1,\dots,y_n$. Let $\mathrm{H}_{a_1,\dots,a_\ell}$ be the finite-dimensional quotient of $\Hdeg_n$ by the two-sided ideal generated by $(y_1-a_1)(y_1-a_2)\cdots (y_1-a_\ell)$ where $(a_1,\dots,a_\ell)\in\K^\ell$. 
For generic $(a_1,\dots,a_\ell)$, the algebra $\mathrm{H}_{a_1,\dots,a_\ell}$ is semisimple and has basis given by the basis elements of $\Hdeg_n$ up to degree $\ell-1$. Let $\pi:\Hdeg_n\rightarrow \mathrm{H}_{a_1,\dots,a_\ell}$ be the quotient map, then $\pi(J(\Hdeg_n))\subseteq \JJ(\mathrm{H}_{a_1,\dots,a_\ell})=0$ and so $\pi(h)=0$ for any generic $\ell$-tuple of complex numbers $(a_1,\dots,a_\ell)$. 
 Taking $\ell$ to be bigger than the degree of $h$ we have $\pi(h)=h$, but for a generic $\ell$-tuple of parameters $(a_1,\dots,a_\ell)$ we have $\pi(h)=0$, so $h=0$.
\end{proof}
\noindent Note that the algebra $\Hdeg_n$, however, still fails to be semisimple as it is not artinian, and its representation theory contains plenty of examples of non-semisimple modules.

\begin{lemma}\label{J rad sv2}
Suppose $n=2$. Then $\JJ(\sv_2)=(e_1)$. 
\end{lemma}
\begin{proof}
Write $e:=e_1$. We have $\JJ(\sv_2)\subseteq\Ker\Phi_2=(e)$ since $\Phi(\JJ(\sv_2))\subseteq \JJ(\Hdeg_2)=0$. On the other hand, Lemma \ref{disentangling} and \cite[Lemma 8]{us+} imply that for any element $X$ of $\sv_2$ it holds that $eXe=0$. Then the element $(1-Xe)$ has left inverse $(1+Xe)$, so $e$ is in the Jacobson radical. Therefore it also holds that $(e)\subseteq \JJ(\sv_2)$ and the statement is proved.
\end{proof}

\begin{lemma}\label{J in e ideal}
Suppose $n>2$. Then $\JJ(\sv_n)\subsetneq (e_1)$.
\end{lemma}
\begin{proof}
We know that $\Phi_n(\JJ(\sv_n))\subseteq \JJ(\Hdeg_n)=0$ by Lemma \ref{dunno}. Therefore $\JJ(\sv_n)\subseteq\Ker\Phi_n=(e_1,\dots,e_n)=(e_1)$. However, $e_i\notin \JJ(\sv_n)$ for all $i=1,\dots,n-2$ since the element $1+e_{i+1}e_i$ is not invertible: multiply $1+e_{i+1}e_i$ on the left by $e_i$ to get $e_i+e_ie_{i+1}e_i=e_i-e_i=0$. Similarly, $e_{n-1}\notin \JJ(\sv_n)$ since $1+e_{n-2}e_{n-1}$ is not invertible (multiply it on the left by $e_{n-1}$ to get $0$).
\end{proof}

\begin{lemma}
Let $n>2$ and suppose $n$ is even. Then $e_1e_3\cdots e_{n-1}\in \JJ(\sv_n)$ and $\JJ(\sBr_n)$.
\end{lemma}
\begin{proof}
Let $e=e_1e_3\cdots e_{n-1}$. Then $e\in \JJ(\sv_n)$ if and only if $(1-Xe)$ has a left inverse for any $X\in\sv_n$. Observe that $eXe$ is a closed loop without any through strands, possibly with complicated self-crossings and a bunch of beads on it. As in Lemma \ref{J rad sv2}, using the relations of Lemma \ref{disentangling} and \cite[Lemma 4-9]{us+}, one can show that $eXe=0$. To do this precisely, do downward induction on the number of beads on the loop (i.e. the total number of $y$'s appearing in the expression), for which the base case is the statement that any closed loop in $\sBr_n$ is $0$. Then do downward induction on the number of crossings. The base case for that is \cite[Lemma 8]{us+}. If a crossing appears directly below a cap or directly above a cup, use Lemma \ref{disentangling} to eliminate the crossing, then the result is $0$ by induction. If the crossing is in between two caps or two cups, move any beads out of the way if necessary across the curve of the cap or cup. These moves produce terms with less beads that are $0$ by induction. Then apply \cite[Lemma 4(a)]{us+} or its upside-down version to switch the position of the crossing and the cup or cap, then apply braid relations if necessary until the crossing can be canceled either using $s_i^2=1$ or using \cite[Lemma 8]{us+}. Either way, the number of crossings has been reduced and the result is $0$ by induction. Therefore $eXe=0$, and so $(1+Xe)(1-Xe)=1-XeXe=1$ showing that $(1+Xe)$ is a left inverse to $(1-Xe)$ for any $X\in\sv_n$. It follows from Lemma \ref{J} that $e\in \JJ(\sv_n)$. Since $e\in\sBr_n$, it also holds that $e\in \JJ(\sBr_n)$.
\end{proof}

\begin{remark}
The reason $n$ is required to be even is that if $n$ is odd, the corresponding element $e:=e_1e_3\cdots e_{n-2}$ (consisting of as many cup-caps placed side by side as will fit) is definitely not in the Jacobson radical. To see this, consider $s:=s_2s_4\cdots s_{n-1}$. Then $ese=\pm e$, so $sese=\pm se$. In the case $sese=se$, $se$ is idempotent so it can't be in $\JJ(\sv_n)$. In the case $sese=-se$, $-se$ is idempotent so it cannot be in $\JJ(\sv_n)$ or $\JJ(\sBr_n)$. In either case, this means $e$ cannot be in $\JJ(\sv_n)$ or $\JJ(\sBr_n)$ since the Jacobson radical is an ideal. We illustrate the elements involved in the case $n=7$:

$$e:=e_1e_3e_5=\TikZ{[scale=.5] \draw
(0,2) arc(-180:0:.5) 
(0,0) arc(180:0:.5)
(2,2) arc(-180:0:.5) 
(2,0) arc(180:0:.5)
(4,2) arc(-180:0:.5) 
(4,0) arc(180:0:.5)
(6,0)node{} to (6,2)node{};}, \quad s:=s_2s_4s_6=\TikZ{[scale=.5] \draw
(0,0)node{} to (0,2)node{}
(1,0)node{} to (2,2)node{}
(2,0)node{} to (1,2)node{}
(3,0)node{} to (4,2)node{}
(4,0)node{} to (3,2)node{}
(5,0)node{} to (6,2)node{}
(6,0)node{} to (5,2)node{};},
$$
and then
$$ses=\TikZ{[scale=.5] \draw
(0,2) arc(-180:0:.5) 
(0,0) arc(180:0:.5)
(2,2) arc(-180:0:.5) 
(2,0) arc(180:0:.5)
(4,2) arc(-180:0:.5) 
(4,0) arc(180:0:.5)
(6,0)node{} to (6,2)node{}
(0,2)node{} to (0,4)node{}
(1,2)node{} to (2,4)node{}
(2,2)node{} to (1,4)node{}
(3,2)node{} to (4,4)node{}
(4,2)node{} to (3,4)node{}
(5,2)node{} to (6,4)node{}
(6,2)node{} to (5,4)node{}
(0,6) arc(-180:0:.5) 
(0,4) arc(180:0:.5)
(2,6) arc(-180:0:.5) 
(2,4) arc(180:0:.5)
(4,6) arc(-180:0:.5) 
(4,4) arc(180:0:.5)
(6,4)node{} to (6,6)node{};}
=- \TikZ{[scale=.5] \draw
(0,2) arc(-180:0:.5) 
(0,0) arc(180:0:.5)
(2,2) arc(-180:0:.5) 
(2,0) arc(180:0:.5)
(4,2) arc(-180:0:.5) 
(4,0) arc(180:0:.5)
(6,0)node{} to (6,2)node{};},
$$ 
where the tangled strand is just pulled straight (with a sign).

\end{remark}


\section{Calibrated representations and Jucys--Murphy elements}\label{cali reps and JM elts} By Definition \ref{def svn} the variables $y_1,\dots,y_n$ commute, and moreover, $\K[y_1,\dots,y_n]$ is a maximal commutative subalgebra of $\sv_n$ \cite{us+}. Similarly, $Y_1,\dots,Y_n$ generate a commutative subalgebra of $\sBr_n$ \cite{Coulembier}. 
In parallel with the theory of Jucys--Murphy elements of the symmetric group, $Y_{j+1}$ commutes with $\sBr_j\subset \sBr_n$, where $\sBr_j$ is the subalgebra generated by $e_1,\dots, e_{j-1}$ and $s_1,\dots, s_{j-1}$  \cite{Coulembier}. Lifting this relation to the affine level, $y_{j+1}$ commutes with $\sv_{j}\subset \sv_n$ for each $j=1,\dots,n-1$, where $\sv_j$ denotes the subalgebra of $\sv_n$ generated by $e_1,\dots, e_{j-1}$, $s_1,\dots,s_{j-1}$, and $y_1,\dots,y_j$ (this follows from Definition \ref{def svn}($\vv_2$) and ($\vv_3)$). 

One way to get a grip on some representations of $\sv_n$ is to study those representations on which $y_1,\dots,y_n$ act semisimply:
\begin{definition} \begin{enumerate} \item A representation $V$ of $\sv_n$ is called calibrated if $V$ has a basis with respect to which $y_j$ acts on $V$ by a diagonal matrix for all $j=1,\dots, n$. \item A representation $V$ of $\sBr_n$ is called calibrated if $V$ has a basis with respect to which $Y_j$ acts on $V$ by a diagonal matrix for all $j=1,\dots, n$.
\end{enumerate}
\end{definition}
\noindent Calibrated representations of $\sBr_n$ are a special case of calibrated representations of $\sv_n$: a representation $V$ of $\sBr_n$ is calibrated if and only $V$ is a calibrated $\sv_n$-representation on which $y_1$ acts by $0$.

\begin{example}
Every $\sBr_n$-representation of the form $\Infl S(\lambda)$ where $S(\lambda)$ is an irreducible $\K S_n$-representation is calibrated: every $e_i$ acts by $0$, $Y_1=0$ by definition, and so for all $2\leq j\leq n$, $Y_j$ acts by  $s_{j-1}Y_{j-1}s_{j-1}+Y_{j-1}=:X_j=\sum\limits_{k=1}^{j-1}(k,j)$, the $j$'th Jucys--Murphy element of $\K S_n$. It is a classical theorem in representation theory of the symmetric group that any irreducible $\K S_n$-representation has a basis with respect to which $X_2,\dots,X_n$ act by diagonal matrices \cite{Jucys},\cite{Murphy}; the foundational work of \cite{OkounkovVershik} builds up the representation theory of $S_n$ from scratch using the Jucys--Murphy elements.

\end{example}

\begin{example}
Every calibrated $\Hdeg_n$-representation is a calibrated $\sv_n$-representation, again by inflation.
\end{example}
\noindent The irreducible calibrated $\Hdeg_n$-representations were first classified by Cherednik \cite{Cherednik} (and see Kriloff and Ram's work \cite{KriloffRam} for the classification of the irreducible calibrated representations of degenerate affine Hecke algebras associated to other finite Coxeter groups).
Since they are known, we are interested in studying those irreducible calibrated $\sv_n$-representations which do not factor through $\Hdeg_n$, so those on which $e_i$ does not act by $0$ for all $1\leq i\leq n-1$.

\subsection{Action of the center of $\sv_n$ on calibrated representations}
Now recall the central element $\Theta\in \JJ(A_n)$ discussed in Section \ref{J rad}. Moving up to $\sv_n$, consider the similar-looking but $S_n$- rather than $S_{n-1}$-symmetric element:
$$\widetilde{\Theta}=\prod_{1\leq i<j\leq n}(1-(y_i-y_j)^2)=(1-(y_1-y_2)^2)(1-(y_1-y_3)^2)\cdots(1-(y_{n-1}-y_n)^2).$$
The element $\widetilde{\Theta}$ belongs to the center of $\sv_n$ \cite[Theorem 53]{us+}. The surjection $\Pi_n:\sv_n\twoheadrightarrow \sBr_n$ sends $\widetilde{\Theta}$ to $\Theta\prod_{2\leq j\leq n} (1-Y_j^2)$, so by Lemma \ref{J}(\ref{J anni}) we know that $\widetilde{\Theta}$ acts by $0$ on any irreducible $\sv_n$-representation which factors through $\sBr_n$. 

For a representation $V$ of $\sv_n$, write $y_j$ for the matrix by which the element $y_j$ acts on $V$ and let $(y_j)_{\ell,m}$ be the matrix entries, and similarly for $e_i$ and $s_i$.

\begin{theorem}\label{Theta} Suppose $V$ is an indecomposable calibrated representation of $\sv_n$ with a nonzero action of some $e_i$, $1\leq i\leq n-1$. Then $\widetilde{\Theta}$ acts by $0$ on $V$.
\end{theorem}

\begin{proof}
Let $d=\dim V$. Let $s_i,e_i,y_j$ denote the matrices of those generators acting on $V$ in the basis of joint eigenvectors $\overline{v}_1,\dots,\overline{v}_d$ for $y_1,\dots, y_n$. We must show that for each $\ell=1,\dots, d$, there exist $j,j'\in\{1,\dots,n-1\}$ such that $(y_j-y_{j'})_{\ell,\ell}=\pm 1$.

By Lemma \ref{cupcap ideal 1}, the matrix $e_i\neq 0$ for all $i=1,\dots, n-1$. As in the first part of the proof of \cite[Theorem 11]{some of us}, straightforward computations using Definition \ref{def svn}($\vv_8$) and ($\vv_7$) show that $(e_i)_{\ell,m}\neq 0$ implies that $(y_i-y_{i+1})_{\ell,\ell}=1$ and $(y_i-y_{i+1})_{m,m}=-1$, while $(s_i)_{\ell,m}\neq 0$ implies that $(y_i-y_{i+1})_{\ell,\ell}=\alpha$ and $(y_i-y_{i+1})_{m,m}=-\alpha$ for some $0\neq \alpha\in\K$.
 Therefore we would be done if for every $k=1,\dots, d$, there existed $i\in\{1,\dots,n\}$ such that the matrix $e_i$ had a nonzero entry in row or column $k$. However, for a general indecomposable representation there is no reason this should hold (see \cite{some of us} for counterexamples to such an expectation when $n=2$), and for the rest of the proof we assume this is not the case.

Order the basis of $y$-eigenvectors as follows: $\overline{v}_1,\dots,\overline{v}_r$ are the basis vectors such that for some $s=1,\dots,r$, there exist $j,j'\in\{1,2,\dots,n-1\}$ such that $(y_j-y_{j'})\overline{v}_s=\pm \overline{v}_s$, i.e., $(y_j-y_{j'})_{s,s}=\pm 1$. We have $r>0$ since every $e_i$ acts by a nonzero matrix, by the remarks above, and furthermore, if $(e_i)_{uv}\neq 0$ then $u,v\in\{1,2,\dots,r\}$. Supposing that $r<d$, we then have $\overline{v}_{r+1},\dots,\overline{v}_d$ with the property that for all $r+1\leq t\leq d$, $(y_j-y_{j'})_{t,t}\neq\pm 1$. Now the representation $V$ is assumed to be indecomposable, but every $e_i$ and every $y_j$ preserve the vector space decomposition $\mathrm{Span}\langle \overline{v}_1,\dots,\overline{v}_r\rangle \oplus \mathrm{Span}\langle \overline{v}_{r+1},\dots,\overline{v}_d\rangle$. . Therefore some $s_i$ must have a nonzero entry $(s_i)_{\ell,m}$ with either (i) $1\leq \ell\leq r$ and $r+1\leq m\leq d$, or (ii) $1\leq m\leq r$ and $r+1\leq \ell\leq d$. 

Without loss of generality let's assume (i) (as for situation (ii), the argument is the same). We will show that this forces $(y_k-y_{k'})\overline{v}_m=\pm \overline{v}_m$ for some $k,k'$, contradicting the assumption that $\overline{v}_1,\dots,\overline{v}_r$ are all of the basis vectors for which this happens.
 Since $1\leq \ell\leq r$, there exist some $k,k'\in\{1,2,\dots,n-1\}$ such that $(y_k-y_{k'})_{\ell,\ell}=\pm 1$. There are three cases: $k,k'\notin\{i,i+1\}$, $\{k,k'\}=\{i,i+1\}$, and $k\in\{i,i+1\}$ but $k'\notin\{i,i+1\}$. 
\underline{Case 1}. Assume $k,k'\notin\{i,i+1\}$. Using that $y_ks_i=s_iy_k$ if $k\neq i,i+1$ (Definition \ref{def svn}($\vv_3$)(iii)) and that $(s_i)_{\ell,m}\neq 0$, we get $(y_k)_{\ell,\ell}=(y_k)_{m,m}$ and $(y_{k'})_{\ell,\ell}=(y_{k'})_{m,m}$ and therefore $(y_k-y_{k'})_{m,m}=(y_k-y_{k'})_{\ell,\ell}=\pm 1$. 
\underline{Case 2}. Suppose $k=i$ and $k'=i+1$. Then $(y_i-y_{i+1})_{\ell,\ell}=\pm1$. But $(s_i)_{\ell,m}\neq 0$ implies that $(y_i-y_{i+1})_{\ell,\ell}=-(y_i-y_{i+1})_{m,m}$, so $(y_i-y_{i+1})_{m,m}=\pm 1$.
\underline{Case 3}. Suppose $k'=i$ and $k\notin\{i,i+1\}$. Then $(y_i-y_k)_{\ell,\ell}=\pm1$ and $(y_k)_{\ell,\ell}=(y_k)_{m,m}$. Considering the $\ell,m$'th matrix entry of the matrix equation $s_iy_i-y_{i+1}s_i=-e_i-1$ yields that $(y_i)_{m,m}=(y_{i+1})_{\ell,\ell}$ since $(e_i)_{\ell,m}=0$ for any $m>r$; similarly, $(y_i)_{\ell,\ell}=(y_{i+1})_{m,m}$ thanks to the equation $y_is_i-s_iy_{i+1}=e_i-1$. Then $(y_{i+1})_{m,m}=(y_i)_{\ell,\ell}=(y_k)_{\ell,\ell}\pm 1=(y_k)_{m,m}\pm 1$. This completes the proof. 
\end{proof}

The center of $\sBr_n$ contains a subalgebra given in terms of $\Theta$ and symmetric polynomials in $Y_2,\dots,Y_n$ \cite[Theorem 6.4.2]{Coulembier}: $$\K\oplus\Theta\K[Y_2,\dots,Y_n]^{S_{n-1}}\subseteq Z(\sBr_n).$$ Similarly, the center of $\sv_n$ is described (completely) in terms of $\widetilde{\Theta}$ and symmetric polynomials in $y_1,\dots,y_n$.
\begin{theorem}\label{Center}\cite[Theorem 53]{us+}
The center $Z(\sv_n)$ of $\sv_n$ is the subalgebra
$$Z(\sv_n)\cong\K\oplus\widetilde{\Theta}\K[y_1,\dots,y_n]^{S_n}.$$
That is, $h\in Z(\sv_n)$ if and only if $h=\widetilde{\Theta}f(y_1,\dots,y_n)+c$ for some symmetric polynomial $f(y_1,\dots,y_n)$ and some $c\in\K$. 
\end{theorem}
\noindent Combined with Theorem \ref{Center}, Theorem \ref{Theta} immediately implies that $Z(\sv_n)$ acts trivially on those indecomposable calibrated representations of $\sv_n$ not factoring through the degenerate affine Hecke algebra:
\begin{corollary}\label{Theta cor} Let $h=\widetilde{\Theta}f(y_1,\dots,y_n)+c\in Z(\sv_n)$. Suppose $V$ is an indecomposable calibrated representation of $\sv_n$ with a nonzero action of some $e_i$, $1\leq i\leq n$. Then $h\cdot v=cv$ for all $v\in V$.
\end{corollary}

\subsection{Eigenvalues of calibrated $\sv_n$-representations}
For any fixed $i$, we have a copy of $\sv_2$ given by $e_i$, $s_i$, $y_i$, and $y_{i+1}$ satisfying the defining relations of Definition \ref{def svn} in the case $n=2$, with $e_1, s_1,y_1,y_2$ replaced by $e_i,s_i,y_i,y_{i+1}$. We will call such a copy $\{e_i,s_i,y_i,y_{i+1}\}$ of $\sv_2$  an \textit{$\sv_2$-quadruple.} In our previous paper \cite{some of us}, we studied the calibrated representations of $\sv_2$. The representations of $\sv_2$ are fundamental for studying calibrated $\sBr_n$-representations for $n>2$, just as the representations of $\Hdeg_2$ play a fundamental role in Okounkov and Vershik's study of $\K S_n$-representations \cite[Section 5]{OkounkovVershik}, or as $\mathfrak{sl}_2$-triples appear as the basic building blocks for higher rank groups in Lie theory (see the Jacobson-Morozov Theorem and its applications, \cite[Theorem 3.7.1 and Section 3.7]{ChrissGinzburg}). 

We state some basic facts about calibrated $\sv_2$-representations which follow easily from the defining relations ($\sv_7$) and ($\sv_8$) in Definition \ref{def svn} and appeared in our paper about indecomposable calibrated representations of $\sv_2$ \cite[Theorem 11]{some of us}:
\begin{lemma}\label{basicbitch}\cite{some of us} Let $\sv_2=\{ e_i,s_i,y_i,y_{i+1}\}$ be an $\sv_2$-quadruple in $\sv_n$, and let $V$ be a calibrated $\sv_2$-representation. The matrices of $e_i$, $s_i$, $y_i$, and $y_{i+1}$ in the basis of $y$-eigenvectors satisfy:
\begin{itemize}
\item $(e_i)_{\ell,\ell}=0$ for all $\ell=1,\dots, \dim V$;
\item $(e_i)_{\ell,m}\neq 0$ implies that $(y_i-y_{i+1})_{\ell,\ell}=1$ and $(y_i-y_{i+1})_{m,m}=-1$;
\item the formula $(e_i)_{\ell,m}=((y_i)_{\ell,\ell}-(y_{i+1})_{m,m})(s_i)_{\ell,m}$ holds for all $\ell\neq m$;
\item $-1=((y_i)_{\ell,\ell}-(y_{i+1})_{\ell,\ell})(s_i)_{\ell,\ell}$.
\end{itemize}
\end{lemma}
\noindent These facts are crucial for understanding the $y_j$- or $Y_j$-eigenvalues of a calibrated $\sv_n$- or $\sBr_n$-representation.

\begin{theorem}\label{integer evalues} Let $V$ be a calibrated $\sv_n$-representation. Suppose the eigenvalues of $y_1$ are in $\Z$. Then the eigenvalues of $y_j$ are in $\Z$ for all $j=2,\dots, n$.
\end{theorem}
\begin{proof} 

Assume by induction that $y_i$ has integer eigenvalues for some $i\geq 1$, so in the basis of eigenvectors for $V$ the matrix $y_i$ is a diagonal $d$ by $d$ matrix with integer entries, $d=\dim V$. Let $1\leq \ell,m\leq d$, $\ell\neq m$, and suppose first that $(e_i)_{\ell,m}\neq 0$. Then $(y_i-y_{i+1})_{\ell,\ell}=1$ and $(y_i-y_{i+1})_{m,m}=-1$ so $(y_{i+1})_{m,m}=(y_i)_{m,m}+1$ and $(y_{i+1})_{\ell,\ell}=(y_i)_{\ell,\ell}-1$ and therefore $(y_{i+1})_{\ell,\ell},(y_{i+1})_{m,m}\in\Z$. If $(e_i)_{\ell,m}= 0$ there are two possibilities. Either (i) $(s_i)_{\ell,m}\neq 0$, or (ii) $(s_i)_{\ell,m}=0$. In case (i), $0=(e_i)_{\ell,m}=((y_i)_{\ell,\ell}-(y_{i+1})_{m,m})(s_i)_{\ell,m}$ implies that $(y_{i+1})_{m,m}=(y_i)_{\ell,\ell}\in\Z$. So $(y_{i+1})_{m,m}\in\Z$ for all $m$ such that column $m$ of $s_i$ contains a nonzero off-diagonal entry. Then case (ii) only needs to be dealt with when there exists some $1\leq m\leq d$ with $(s_i)_{\ell,m}=0$ for all $\ell\neq m$. Since $(s_i)^2=1$, we must have $(s_i)_{m,m}=\pm 1$. The formula $-1=((y_i)_{m,m}-(y_{i+1})_{m,m})(s_i)_{m,m}$ then gives us that $(y_i)_{m,m}-(y_{i+1})_{m,m}=\mp 1$, so again $(y_{i+1})_{m,m}\in\Z$ by induction.
\end{proof}

In fact, the proof just given tells us more than just that the eigenvalues of $y_j$ are in $\Z$. It also tells us that if $e_i$ has nonzero entries in all the off-diagonal spots where $s_i$ has nonzero entries (i.e. $(e_i)_{\ell,m}\neq 0$ whenever $(s_i)_{\ell,m}\neq 0$ for $\ell\neq m$) then each diagonal entry of $y_{i+1}$ is obtained from the corresponding diagonal entry of $y_i$ by adding or subtracting $1$ (i.e. $y_{i+1}\overline{v}_m=(y_i\pm 1)\overline{v}_m$ for all $m=1,\ldots,\dim V$).

Let us consider in more detail the case that $(e_i)_{\ell,m}=0$ but $(s_i)_{\ell,m}\neq 0$ in a calibrated representation.  Let $(y_i)_{\ell,\ell}=c$ and $(y_i)_{m,m}=d$. If we consider the $\K$-vector subspace of $V$ spanned by $\overline{v}_\ell$ and $\overline{v}_m$, then the $\sv_2$-quadruple acting on this subspace is the same as $\Hdeg_2$ acting on this subspace, since $e_i$ acts by $0$. The representations of $\Hdeg_2$ were analyzed in \cite[Section 5]{OkounkovVershik}. In matrices we have:
$$y_i=\begin{pmatrix}
c&0\\0&d
\end{pmatrix},\quad
y_{i+1}=\begin{pmatrix}
d&0\\0&c
\end{pmatrix},\quad
s_i=\begin{pmatrix}
\frac{-1}{c-d}&s_{\ell,m}\\
s_{m,\ell}&\frac{1}{c-d}
\end{pmatrix},\quad
e_i=\begin{pmatrix}
0&0\\0&0
\end{pmatrix}.
$$
In particular, $c\neq d$, and using the equations coming from the off-diagonal entries in the matrix equation $s_i^2=1$ and solving for $s_{m,\ell}$, it follows that $s_{m,\ell}\neq 0$ as well. As an $\Hdeg_2$-representation, this subspace spanned by $\overline{v}_\ell$ and $\overline{v}_m$ is irreducible. Summarizing, we have:

\begin{theorem}\label{evalues must move}
Let $V$ be a calibrated $\sv_n$-representation and let $\overline{v}$ be one of the eigenvectors in the basis for $V$. Then $y_i\overline{v}\neq y_{i+1}\overline{v}$ for all $i=1,\dots,n-1$. 
\end{theorem}

\subsection{Eigenvalues of calibrated $\sBr_n$-representations}
Applying Theorems \ref{integer evalues} and \ref{evalues must move} to the case $y_1=0$, we obtain:
\begin{corollary}\label{evalues sBr ints}
Let $V$ be a calibrated $\sBr_n$-representation. Then the eigenvalues of $Y_j$ are in $\Z$ for all $j=2,\dots, n$, and their entries as diagonal matrices satisfy $(Y_j)_{\ell,\ell}\neq (Y_{j+1})_{\ell,\ell}$ for all $\ell=1,\dots, \dim V$ and for each $j=1,\dots,n-1$.
\end{corollary}
Note that when $V$ is a calibrated $\sBr_n$-representation, $Y_2=s_1Y_1s_1+s_1+e_1=s_1+e_1$ and $Y_2^2=(s_1+e_1)^2=s_1^2+s_1e_1+e_1s_1+e_1^2=\Id-e_1+e_1+0=\Id$, so 
the first step of determining $Y_2$ produces a diagonal matrix with all $1$'s and $-1$'s for its diagonal entries. Since $e_1$ has $0$'s on the diagonal, $Y_2$ is equal to the diagonal of $s_1$ and so all diagonal entries of $s_1$ are $\pm 1$. If we order the basis so that the $-1$ diagonal entries of $s_1$ come first, then we have:
$$s_1=\begin{pmatrix} -\Id & S\\ 0&\Id\end{pmatrix},\quad e_1=\begin{pmatrix} 0&-S\\0&0\end{pmatrix},\quad Y_2=\begin{pmatrix}-\Id&0\\0&\Id\end{pmatrix}.$$

In particular, the non-zero, off-diagonal entries of $s_1$ coincide with the non-zero entries of $e_1$. In general the entries of $e_i$ will not be identical to the off-diagonal entries of $s_i$, but in the situation that for every $i=1,\dots,n-1$, $e_i$ has non-zero entries in all of the spots where $s_i$ has off-diagonal, non-zero entries then the eigenvalues of $Y_2,\dots,Y_n$ are easy to determine. This situation will be illustrated by the class of examples in Section \ref{irreduciblecali}.

\begin{lemma}\label{evalues closed form}
Let $V$ be a calibrated representation of $\sBr_n$. Suppose that for all $i=1,\dots, n-1$ and all $1\leq \ell\neq m\leq \dim V$, $(e_i)_{\ell,m}\neq 0$ if and only if $(s_i)_{\ell,m}\neq 0$. Then $$(Y_{j+1})_{m,m}=(Y_j)_{m,m}+(s_j)_{m,m}=(Y_j)_{m,m}\pm 1$$
and we have the closed formula
$$(Y_{j+1})_{m,m}=\sum_{a=1}^j (s_a)_{m,m}=\sum_{a=1}^j \epsilon_a$$
where $\epsilon_a\in\{\pm 1\}$ for each $a=1,\dots,j$.
\end{lemma}
\begin{proof}
Lemma \ref{basicbitch} implies that $(e_i)_{\ell,m}\neq 0$ only if  $(s_i)_{\ell,m}\neq 0$, so the additional assumption we are imposing is really only the converse. Suppose that $(e_i)_{\ell,m}\neq 0$ if $(s_i)_{\ell,m}\neq 0$. By Lemma \ref{basicbitch}, $(e_i)_{\ell,m}\neq 0$ implies that $(y_i-y_{i+1})_{\ell,\ell}=1$ and $(y_i-y_{i+1})_{m,m}=-1$. Therefore if $e_i$ has nonzero entries in row $\ell$ then it has all $0$ entries in column $\ell$, and if $e_i$ has nonzero entries in column $m$ then it has all $0$ entries in row $m$. Then the same is true for the off-diagonal entries of $s_i$ by our assumption. Let $d=\dim V$. Using that $s_i^2=1$, this implies $1=(s_i^2)_{\ell,\ell}=\sum_{m=0}^d(s_i)_{\ell,m}(s_i)_{m,\ell}=(s_i)_{\ell,\ell}^2$, and therefore $(s_i)_{\ell,\ell}=\pm 1$ for all $\ell=1,\dots,d$. We compute the equation $y_{j+1}=s_jy_js_j+s_j+e_j$ for the $\ell,\ell$'th matrix entry, and using the same property of the rows and columns of $s_j$ and the fact that $e_j$ is $0$ on the diagonal, we find:
\begin{align*}(y_{j+1})_{\ell,\ell}&=\sum_{m=1}^d (s_j)_{\ell,m}(y_j)_{m,m}(s_{j})_{m,\ell}+(s_j)_{\ell,\ell}=(s_j)_{\ell,\ell}^2(y_j)_{\ell,\ell}+(s_j)_{\ell,\ell}\\&=(y_j)_{\ell,\ell}+(s_j)_{\ell,\ell}=(y_j)_{\ell,\ell}\pm 1.
\end{align*}
The closed formula then follows by induction.
\end{proof}

\subsection{Restriction of irreducible calibrated representations and non-semisimplicity} One might expect based upon experience with more familiar algebras that an irreducible calibrated $\sBr_n$-representation would restrict to a direct sum of irreducible calibrated $\sBr_{n-1}$-representations. That is, one can ask the question:
\begin{question} Let $\LL$ be an irreducible calibrated representation of $\sBr_n$ and consider its restriction $\Res^{\sBr_n}_{\sBr_{n-1}}\LL$ where $\sBr_{n-1}$ is the subalgebra of $\sBr_n$ generated by $e_1,\dots,e_{n-2}, s_1,\dots,s_{n-2}$. Clearly $\Res^{\sBr_n}_{\sBr_{n-1}}\LL$ is calibrated since $Y_2,\dots,Y_{n-1}$ act by diagonal matrices on $\LL$ given that $Y_2,\dots,Y_{n-1},Y_n$ all do. Does $\Res^{\sBr_n}_{\sBr_{n-1}}\LL$ decompose as a direct sum of irreducible calibrated representations of $\sBr_{n-1}$?
\end{question}
The whole idea of calibrated representations is that of an inductively defined subalgebra acting semisimply so the answer feels like it should be yes. Indeed, in parts of the literature ``calibrated" is used synonymously with ``completely splittable," where the latter term means exactly that the restriction of the representation to any Levi or Young subgroup is semisimple. Even in positive characteristic, calibrated and completely splittable representations of $\Hdeg_n$ (and hence $k S_n$) coincide \cite[Theorem 2.13]{Ruff}. But for $\sBr_n$, what is the answer?
\begin{answer}
No. And the smallest counterexample is $2$-dimensional. 
Let $\LL$ be the $2$-dimensional representation of $\sBr_3$ given by:
\begin{align*}
&s_1=\begin{pmatrix}-1&-1\\0&1 \end{pmatrix},
\qquad\; s_2=\begin{pmatrix}
1&0\\-1&-1
\end{pmatrix},\\
&e_1=\begin{pmatrix}
0&1\\0&0
\end{pmatrix},\qquad\qquad
e_2=\begin{pmatrix}
0&0\\
-1&0
\end{pmatrix},\\
&y_1=\begin{pmatrix}
0&0\\0&0
\end{pmatrix}=y_3,\qquad
y_2=\begin{pmatrix}
-1&0\\0&1
\end{pmatrix}.
\end{align*}
It is easy to check manually that $\LL$ is irreducible as an $\sBr_3$-representation, and indecomposable as a calibrated $\sBr_2$-representation under the action of $s_1,e_1,y_1,$ and $y_2$. And this is no isolated counterexample, it is the first in an infinite series constructed in the next section: see Theorem \ref{res special}.
\end{answer}

\section{Construction of the irreducible calibrated representations of $\sBr_n$ and $\sv_n$ given by exterior powers of the standard representation of $S_n$}\label{irreduciblecali}

\subsection{The standard representation of $S_n$}

We showed in \cite{some of us} that $e_1$ must act by $0$ on any irreducible calibrated representation of $\sv_2$. This is no longer true for $\sv_n$ when $n>2$: then $\sv_n$ has a natural $(n-1)$-dimensional irreducible calibrated representation $\CC_\alpha(V_n)$, depending on a parameter $\alpha\in\K$, on which $e_i$ acts by a nonzero matrix for all $i=1,\dots,n-1$ and which is irreducible as an $S_n$-representation -- the underlying $S_n$-representation is the standard representation $V_n$.  When $\alpha=0$ we then $\CC_0(V_n)=\Infl \LL(\mu)$ is the inflation from $\sBr_n$ to $\sv_n$ of some irreducible representation of $\sBr_n$ labeled by  $\mu$, a partition of $n$ or $n-2$ or $n-4$, or so on. The reader might be tempted to guess that $\mu=(n-1,1)$; this is incorrect as $\LL(\mu)$ for $\mu$ a partition of $n$ is inflated from $\K S_n$ to $\sBr_n$ and so every $e_i$ acts on it by $0$.

 Here are the matrices for the action of the generators of $\sv_n$ on the representation $\CC_\alpha(V_n)$ when $n=5$, to give an idea. Then we write the formula for arbitrary $n$ below. 
The $s_i$'s act as:
\begin{align*}
s_1&=\small 
\begin{pmatrix} -1&-1&0&0\\0&1&0&0\\0&0&1&0\\0&0&0&1\end{pmatrix}, \quad s_2=\small 
\begin{pmatrix} 1&0&0&0\\-1&-1&-1&0\\0&0&1&0\\0&0&0&1\end{pmatrix}, \\
s_3&=\small 
\begin{pmatrix} 1&0&0&0\\0&1&0&0\\0&-1&-1&-1\\0&0&0&1\end{pmatrix}, \quad s_4=\small 
\begin{pmatrix} 1&0&0&0\\0&1&0&0\\0&0&1&0\\0&0&-1&-1\end{pmatrix}.
\end{align*}
The $e_i$'s act as:
\begin{align*}
e_1&=\small 
\begin{pmatrix} 0&1&0&0\\0&0&0&0\\0&0&0&0\\0&0&0&0\end{pmatrix},\quad e_2=\small 
\begin{pmatrix} 0&0&0&0\\-1&0&1&0\\0&0&0&0\\0&0&0&0\end{pmatrix},\\
e_3&=\small 
\begin{pmatrix} 0&0&0&0\\0&0&0&0\\0&-1&0&1\\0&0&0&0\end{pmatrix},\quad e_4=\small 
\begin{pmatrix} 0&0&0&0\\0&0&0&0\\0&0&0&0\\0&0&-1&0\end{pmatrix}.
\end{align*}
Then we require that $y_1$ acts by 
$$y_1=\small  
\begin{pmatrix} a&0&0&0\\0&a&0&0\\0&0&a&0\\0&0&0&a\end{pmatrix}$$
for an arbitrary $\alpha\in\K$. Since for each $i=2,3,4,5$, $y_i$ can be expressed in terms of $y_1$, the $e_i$'s and the $s_i$'s, the matrices for $y_2,y_3,y_4,y_5$ are now completely determined. If we compute them using the formula $y_{i+1}=s_iy_is_i+s_i+e_i$ we find:
\begin{align*}
y_2&=\small 
\begin{pmatrix}\alpha-1&0&0&0\\0&\alpha+1&0&0\\0&0&\alpha+1&0\\0&0&0&\alpha+1\end{pmatrix},\quad 
y_3=\small 
\begin{pmatrix}\alpha&0&0&0\\0&\alpha&0&0\\0&0&\alpha+2&0\\0&0&0&\alpha+2\end{pmatrix},\\
y_4&=\small 
\begin{pmatrix}\alpha+1&0&0&0\\ 0&\alpha+1&0&0\\0&0&\alpha+1&0\\0&0&0&\alpha+3\end{pmatrix}, \quad 
y_5=\small 
\begin{pmatrix}\alpha+2&0&0&0\\0&\alpha+2&0&0\\0&0&\alpha+2&0\\0&0&0&\alpha+2\end{pmatrix}. 
\end{align*}
Now we give the general formula for $\CC_\alpha(V_n)$ as a representation for $\sv_n$, generalizing the $n=5$ example above. Let $f_{i,j}$ be the $n-1$ by $n-1$ matrix with all $0$ entries except a $1$ in row $i$, column $j$.

\begin{theorem}\label{reflection rep} Let $\CC_\alpha(V_n)$ be the $(n-1)$-dimensional $\K$-vector space $V_n$ acted on by the following matrices in $\End(V_n)$:
\begin{align*}
&e_1=f_{1,2},\quad e_{n-1}=-f_{n-1,n-2},\;\hbox{ and }\;e_i=-f_{i,i-1}+f_{i,i+1}\hbox{ for }2\leq i\leq n-2,\\
&s_1=-f_{1,1}-f_{1,2}+\sum_{j=2}^{n-1}f_{j,j},\quad s_{n-1}=-f_{n-1,n-2}-f_{n-1,n-1}+\sum_{j=1}^{n-2}f_{j,j},\\
&\quad \hbox{ and }\; s_i=-f_{i,i-1}-f_{i,i}-f_{i,i+1}+\sum_{j=1}^{i-1}f_{j,j}+\sum_{k=i+1}^{n-1}f_{k,k} \hbox{ for }2\leq i\leq n-2,\\
&y_1=\alpha\sum_{j=1}^{n-1}f_{j,j},\quad y_{i+1}=y_{i}-f_{i,i}+\sum_{j=1}^{i-1}f_{j,j}+\sum_{k=i+1}^{n-1}f_{k,k} \hbox{ for } 1\leq i\leq n-1.
\end{align*}
Then 
the matrices $e_i,s_i,y_j$ satisfy the defining relation of $\sv_n$ given in Definition \ref{def svn}. 
\end{theorem}
The proof that $\CC_\alpha(V_n)$ is an $\sv_n$-representation is subsumed by the proof of Theorem \ref{exterior powers}. A Mathematica code for checking that $V_n$ is an $\sv_n$-representation for a given value of $n$ is in the Appendix \ref{code}. For the purposes of setting up 
Theorem \ref{exterior powers}, we check:
\begin{lemma}\label{checking Vn a rep of Sn} $V_n$ is a representation of $S_n$. \end{lemma}\begin{proof} We need to check that $s_i^2=1$, $s_is_{i+1}s_i=s_{i+1}s_is_{i+1}$, and $s_is_j=s_js_i$ for $|i-j|>1$.

Write $s_i=t_i+p_i$, where $t_i=\diag(s_i)$ is the diagonal part of $s_i$ and $p_i=s_i-t_i$ is the off-diagonal part of $s_i$. Then $t_i=\diag(1,1,\dots,1,-1,1,\dots,1)$ with $-1$ in position $i$, and for $1\leq i\leq n-2$, $p_i=-f_{i,i-1}-f_{i,i+1}$ while $p_1=-f_{1,2}$ and $p_{n-1}=-f_{n-1,n-2}$. Then $t_i^2=1$, $p_i^2=0$, $t_ip_i=-p_i$, and $p_it_i=p_i$, so $s_i^2=(t_i+p_i)^2=1$.

Next we check the braid relation. Compute \begin{align*}s_is_{i+1}&=\sum_{j=1}^{i-1}f_{j,j}+\sum_{k=i+2}^{n-1}f_{k,k}-f_{i,i-1}+f_{i,i+1}+f_{i,i+2}-f_{i+1,i}-f_{i+1,i+1}-f_{i+1,i+2}\\
&=\diag(1,\dots,1,0,-1,1,\dots,1)+\left(-f_{i,i-1}-f_{i+1,i+2}+f_{i,i+1}+f_{i,i+2}-f_{i+1,i}\right), 
\end{align*}
where the $0$ is in row $i$. Let $1\leq i\leq n-2$ and in the extremal cases $i=1$ or $i+1=n-1$ set $f_{0,-}$ and $f_{-,n}$ equal to $0$ in the formulas.
Then 
\begin{align*}
s_is_{i+1}s_i&=s_is_{i+1}(t_i+p_i)\\
&=\diag(1,\dots,1,0,-1,1,\dots,1)+\left(-f_{i,i-1}-f_{i+1,i+2}+f_{i,i+1}+f_{i,i+2}+f_{i+1,i}\right)\\
&\quad +0 + \left(f_{i+1,i-1}+f_{i+1,i+1}\right)\\
&=\diag(1,\dots,1,0,0,\dots,1)\\&\qquad+\left(f_{i+1,i-1}-f_{i,i-1}+f_{i+1,i}+f_{i,i+1}-f_{i+1,i+2}+f_{i,i+2}  \right)\\
s_{i+1}s_is_{i+1}&=(t_{i+1}+p_{i+1})s_is_{i+1}\\
&=\diag(1,\dots,1,0,1,1,\dots,1)+\left(-f_{i,i-1}+f_{i+1,i+2}+f_{i,i+1}+f_{i,i+2}+f_{i+1,i}\right)\\
&\quad -f_{i+1,i+2}+\left(f_{i+1,i-1}-f_{i+1,i+1}-f_{i+1,i+2}\right)\\
&=\diag(1,\dots,1,0,0,1,\dots,1)\\
&\qquad +\left(f_{i+1,i-1}-f_{i,i-1}+f_{i+1,i}+f_{i,i+1}-f_{i+1,i+2}+f_{i,i+2}\right).
\end{align*}

Finally, the relation $s_is_j=s_js_i$ if $|i-j|>1$ is clear.
\end{proof}
\begin{theorem}\label{checking it's the reflection rep}
The representation $\CC_\alpha(V_n)$ is irreducible considered as a $\K S_n$-representation, and therefore is irreducible as an $\sv_n$-representation (and as an $\sBr_n$-representation if $\alpha=0$). As a $\K S_n$-representation, $\CC_\alpha(V_n)\cong S(n-1,1)$, the standard representation. 
\end{theorem}
\begin{proof} Let $\overline{v}=(v_1,\dots,v_{n-1})\in V_n$, $\overline{v}\neq 0$. We must show that $\K S_n\cdot \overline{v}=V_n$. It is clear that any vector $(0,\dots,\star,\dots,0)\in V_n$ with a single nonzero entry $\star$ generates $V_n$ as a $\K S_n$-representation, so if $\K S_n\cdot\overline{v}$ contains such a vector then we are done. Let $j\in\N$ be minimal such that $v_j\neq 0$. If $j>1$ then $s_{j-1}\cdot \overline{v}=\overline{v}-(0,\dots,v_j,\dots, 0)$ with $v_j$ in the $j-1$'st spot, so then $(s_{j-1}-1)\cdot \overline{v}=(0,\dots,\star,\dots,0)$ and we are done. So we many assume that $v_1\neq 0$. Then applying $s_1-1$ to $\overline{v}$ we get that $(-2v_1-v_2,0,\dots,0)$ is in $\K S_n\cdot \overline{v}$. So if $v_2\neq -2v_1$ then $\K S_n\cdot \overline{v}=V_n$. Suppose $v_2=-2v_1$. Applying $s_2-1$ to $\overline{v}$ we then get that $(0,3v_1-v_3,0,\dots,0)$ is in $\K S_n\cdot \overline{v}$ so we are done unless $v_3=3v_1$. Continuing in this way by applying $(s_i-1)$ for $i=1,\dots,n-2$ in succession, we see that if $\K S_n\cdot\overline{v}\subsetneq V_n$ then $\overline{v}=(v_1,-2v_1,3v_1,-4v_1,\dots,(-1)^{n-1}(n-1)v_1)$. If this is the case, then consider $(s_{n-1}-1)\overline{v}=(0,\dots,0,(-1)^{n-1}(n-2)v_1+(-1)^n2(n-1)v_1)=(0,\dots,0,(-1)^{n}n v_1)\neq 0$, which then generates $V_n$. So $\K S_n\cdot \overline{v}=V_n$. 

Now to see which irreducible representation of $\K S_n$ this representation $V_n$ is, we use the character $\chi_{V_n}$ of $V_n$. We have $\chi_{V_n}(1)=n-1$ and $\chi_{V_n}(s_1)=n-3$.  The only irreducible character $\chi$ of $\K S_n$ with these properties is $\chi_{(n-1,1)}$, the character of the standard representation. 
\end{proof}

\begin{corollary}
Let $\alpha=0$. We then have $\CC_0(V_n)=\Infl\LL(\mu)$ for some $\mu\in\Lambda$ and some irreducible $\sBr_n$-module $\LL(\mu)$. It follows that the Jucys--Murphy elements $Y_j=\Pi(y_j)$ act on $\LL(\mu)$ by the diagonal matrices given by Theorem \ref{reflection rep} with $\alpha=0$.
\end{corollary}

\subsection{All irreducible representations of $\sBr_4$ are calibrated}\label{A4}
Let $n=4$. There are three irreducible representations of $S_4$ of dimension bigger than $1$, labeled by the partitions $(3,1),(2^2),(2,1^2)$. It turns out that we may extend two of these three irreducible $S_n$-representations in a unique way to a representation of $\sBr_4$ over $\K$ such that the $e_i$'s act by nonzero matrices and the $y_j$'s act diagonally with $y_1$ acting by $0$. 

$\bullet\;$ $\CC_0(3,1)=\CC_0(V_4)$ is calibrated by Theorem \ref{reflection rep}. Explicitly, the matrices are given by:
\begin{align*}
&s_1=\small 
\begin{pmatrix}
-1&-1&0\\0&1&0\\0&0&1
\end{pmatrix},\quad
s_2=\small  
\begin{pmatrix}
1&0&0\\-1&-1&-1\\0&0&1
\end{pmatrix},\quad
s_3=\small 
\begin{pmatrix}
1&0&0\\0&1&0\\0&-1&-1
\end{pmatrix},\\
&e_1=\small 
\begin{pmatrix}
0&1&0\\0&0&0\\0&0&0
\end{pmatrix},\quad
e_2=\small 
\begin{pmatrix}
0&0&0\\-1&0&1\\0&0&0
\end{pmatrix},\quad
e_3=\small 
\begin{pmatrix}
0&0&0\\0&0&0\\0&-1&0
\end{pmatrix},\\
&y_1=\small 
\begin{pmatrix}
0&0&0\\0&0&0\\0&0&0
\end{pmatrix},\quad
y_2=\small 
\begin{pmatrix}
-1&0&0\\0&1&0\\0&0&1
\end{pmatrix},\quad
y_3=\small 
\begin{pmatrix}
0&0&0\\0&0&0\\0&0&2
\end{pmatrix},\quad
y_4=\small 
\begin{pmatrix}
1&0&0\\0&1&0\\0&0&1
\end{pmatrix}. 
\end{align*}
\\

$\bullet\;$ 
 The $S_4$-representation $S(3,1)$ is the standard representation $V_4$ and $S(2,1^2)=\Lambda^2 V_4$. We may then calculate the matrices for $s_i$ acting on $\Lambda^2 V_4$ :
\begin{align*}
&s_1=\small  
\begin{pmatrix}
-1&0&0\\0&-1&-1\\0&0&1
\end{pmatrix},\quad
s_2=\small  
\begin{pmatrix}
-1&-1&0\\0&1&0\\0&-1&-1
\end{pmatrix},\quad
s_3=\small  
\begin{pmatrix}
1&0&0\\-1&-1&0\\0&0&-1
\end{pmatrix}.
\end{align*}
It turns out that using the same kind of formulas as we did for $V_n$ for the $e_i$'s and $y_j$'s results in a representation $\CC_0(2,1^2)$ of $\sv_4$ and $\sBr_4$, that is we take:
\begin{align*}
&e_1=\small 
\begin{pmatrix}
0&0&0\\0&0&1\\0&0&0
\end{pmatrix},\quad
e_2=\small 
\begin{pmatrix}
0&1&0\\0&0&0\\0&-1&0
\end{pmatrix},\quad
e_3=\small 
\begin{pmatrix}
0&0&0\\-1&0&0\\0&0&0
\end{pmatrix},\\
&y_1=\small 
\begin{pmatrix}
0&0&0\\0&0&0\\0&0&0
\end{pmatrix},\quad
y_2=\small 
\begin{pmatrix}
-1&0&0\\0&-1&0\\0&0&1
\end{pmatrix},\quad
y_3=\small 
\begin{pmatrix}
-2&0&0\\0&0&0\\0&0&0
\end{pmatrix},\quad
y_4=\small 
\begin{pmatrix}
-1&0&0\\0&-1&0\\0&0&-1
\end{pmatrix}. 
\end{align*}
The trick for $e_i$ is to take $+1$'s above the diagonal and $-1$'s below the diagonal in the same positions as the nonzero non-diagonal entries of $s_i$, and $0$'s everywhere else. To get $y_{j+1}$ we add the diagonal entries of $y_j$ and $s_j$.

$\bullet\;$ It remains to consider the partition $(2^2)$. The underlying irreducible $S_4$-representation is $2$-dimensional. We would be forced by the braid relations, by $s_i^2=1$ and $e_i^2=0$, by $e_ie_{i\pm 1}e_i=-e_i$, and the inductive formula for $y_j$, to choose matrices for the generators of $\sv_4$ up to conjugation by:
\begin{align*}
&s_1=\small 
\begin{pmatrix}
-1&-1\\0&1
\end{pmatrix},\quad
s_2=\small 
\begin{pmatrix}
1&0\\-1&-1
\end{pmatrix},\quad
s_3=s_1, \\
&e_1=\small 
\begin{pmatrix}
0&1\\
0&0
\end{pmatrix},\quad
e_2=\small 
\begin{pmatrix}
0&0\\
-1&0
\end{pmatrix},\quad
e_3=e_1, \\
&y_1=\diag(0,0),\quad y_2=\diag(-1,1),\quad y_3=y_1,\quad y_4=y_2. 
\end{align*}
However, this is not an $\sv_4$-representation because then other relations fail: $s_1e_3\neq e_3s_1$ and $e_2y_4\neq y_4e_2$, for example. So $(2^2)$ does not yield an $\sv_4$- or $\sBr_4$-representation with nonzero action of the $e_i$'s. \\

We know from \cite{Coulembier} that the irreducible representations of $\sBr_4$ are in bijection with $\{(4),(3,1),(2^2),(2,1^2),(1^4),(2),(1^2)\}$. The irreducible representation $\LL(\lambda)$ for $\lambda$ a partition of $4$ is given by the irreducible representation $S(\lambda)$ of $\K S_n$ with $e_i$ acting by $0$ for all $i=1,2,3$; the Jucys--Murphy elements $Y_j$ then act on these representations by the Jucys--Murphy elements of $\K S_4$ so they act semisimply. It must therefore hold that $\{\CC_0(3,1),\CC_0(2,1^2)\}=\{\LL(2),\LL(1^2)\}$. We conclude that $Y_1,Y_2,Y_3,Y_4$ act semisimply on every irreducible representation of $\sBr_4$.

\subsection{The exterior powers of the standard representation of $S_n$}
Theorem \ref{reflection rep} and the $n=4$ examples of Section \ref{A4} generalize to all exterior powers of the standard representation of $S_n$, providing an explicit construction of a family of irreducible calibrated $\sBr_n$-representations. For $1\leq k\leq n-2$, the $\K S_n$-representation $\Lambda^k V_n$ extends to a calibrated representation of $\sBr_n$ and $\sv_n$ on which the $e_i$'s act by nonzero matrices. Although the underlying $S_n$-representation is labeled by the partition $(n-k,1^k)$, the representations constructed in the following theorem {\em do not} coincide with the irreducible $\sBr_n$-representations labeled $\LL(n-k,1^k)$ in the classification given in \cite{Coulembier} of irreducible $\sBr_n$-modules.
\begin{theorem}\label{exterior powers}
Let $V_n$ be the $(n-1)$-dimensional standard representation of $S_n$ and let $1\leq  k\leq n-2$. 
 Then the irreducible $S_n$-representation $\Lambda^k V_n$ admits a nonzero action by every $e_i$ and a semisimple action by every $y_j$ for $1\leq i\leq n-1$ and $1\leq j\leq n$, with $y_1$ acting by the scalar matrix $\alpha\Id$ for some $\alpha\in\K$, endowing $\Lambda^k V_n$ with the structure of an irreducible calibrated $\sv_n$-representation which we denote by $\CC_\alpha(\Lambda^k V_n)$. When $\alpha=0$ then $\CC_0(\Lambda^k V_n)$ is an irreducible calibrated $\sBr_n$-representation .
\end{theorem}

\begin{proof}
The $\K S_n$-representation $\Lambda^k V_n$ is irreducible with natural basis given by the wedges $\mathcal{B}_k:=\{\overline{v}_{a_1}\wedge \overline{v}_{a_2}\wedge\dots\wedge \overline{v}_{a_k}\mid 1\leq a_1<a_2<\ldots<a_k\leq n-1\}$ where $\{\overline{v}_1,\ldots,\overline{v}_{n-1}\}$ is the basis of $V_n$ used in Theorem \ref{reflection rep}; the dimension of $\Lambda^kV$ is ${n-1\choose k}$.  We order the basis elements  lexicographically. Write the $(\ell,m)$'th matrix entry of $X$ by $X(\ell,m)$ instead of $X_{\ell,m}$ so that it is more visible. The representation $\CC_a(\Lambda^k V_n)$ is the ${n-1\choose k}$-dimensional $\K$-vector space equipped with the following actions of the generators of $\sv_n$, given as matrices in the basis $\mathcal{B}_k$.

\underline{The matrix for $s_i$}.\\
 The matrix for $s_i$ acting on $\Lambda^k V_n$ for $1\leq i\leq n-1$ has $0$'s except for the following nonzero entries:
\begin{align*}
&\hbox{Diagonal entries: }\\&s_i(a_1\wedge\dots\wedge a_k,\;a_1\wedge\dots\wedge a_k )=\begin{cases}-1\quad\hbox{ if }a_m=i\hbox{ for some }1\leq m\leq k, \\ \:\:\:\:1\quad\hbox{ otherwise. } \end{cases}\\
&\hbox{Above-diagonal entries: }\\
& s_i(i\wedge a_2\wedge\dots\wedge a_k, \;(i+1)\wedge a_2\wedge\dots\wedge a_k)=-1\quad \mbox{for } i+2\leq a_2<\ldots<a_k, \\
& s_i(a_1\wedge i\wedge\dots\wedge a_k,\;a_1\wedge (i+1)\wedge\dots\wedge a_k)=-1\quad \mbox{for }a_1\leq i-1,\;i+2\leq a_3<\ldots<a_k,\\
&\qquad \qquad\qquad \qquad \vdots\\
& s_i(a_1\wedge a_2\wedge\dots\wedge i,\;a_1\wedge a_2\wedge\dots\wedge (i+1))=-1\quad \mbox{for }a_1<a_2<\ldots<a_{k-1}\leq i-1.\\
&\hbox{Below-diagonal entries: }\\
&s_i(i\wedge a_2\wedge\dots\wedge a_k,\;(i-1)\wedge a_2\wedge\dots\wedge a_k)=-1\quad \mbox{for } i+1\leq a_2<\ldots<a_k, \\
&s_i(a_1\wedge i\wedge\dots\wedge a_k,\;a_1\wedge(i-1) \wedge\dots\wedge a_k)=-1\quad \mbox{for }a_1\leq i-2,\;i+1\leq a_3<\ldots<a_k, \\
&\qquad \qquad\qquad \qquad \vdots\\
&s_i(a_1\wedge a_2\wedge\dots\wedge i ,\;a_1\wedge a_2\wedge\dots\wedge (i-1))=-1\quad\mbox{for } a_1<a_2<\ldots<a_{k-1}\leq i-2.\\
\end{align*}

\underline{The matrix for $e_i$}.\\
The matrix for $e_i$ acting on $\Lambda^k V_n$ for $1\leq i\leq n-1$ has $0$'s on the diagonal, the same entries as $s_i$ below the diagonal, and minus the entries of $s_i$ above the diagonal. That is, all the entries of $e_i$ are $0$ except for:\\
\begin{align*}
&\hbox{Above-diagonal entries: }\\
& e_i(i\wedge a_2\wedge\dots\wedge a_k,\;(i+1)\wedge a_2\wedge\dots\wedge a_k)=1\quad\mbox{for } i+2\leq a_2<\ldots<a_k, \\
& e_i(a_1\wedge i\wedge\dots\wedge a_k,\;a_1\wedge (i+1)\wedge\dots\wedge a_k)=1\quad \mbox{for }a_1\leq i-1,\;i+2\leq a_3<\ldots<a_k, \\
&\qquad \qquad\qquad \qquad \vdots\\
& e_i(a_1\wedge a_2\wedge\dots\wedge i,\;a_1\wedge a_2\wedge\dots\wedge (i+1))=1\quad \mbox{for }a_1<a_2<\ldots<a_{k-1}\leq i-1. \\
&\hbox{Below-diagonal entries: }\\
&e_i(i\wedge a_2\wedge\dots\wedge a_k,\;(i-1)\wedge a_2\wedge\dots\wedge a_k)=-1\quad \mbox{for }i+1\leq a_2<\ldots<a_k, \\
&e_i(a_1\wedge i\wedge\dots\wedge a_k,\;a_1\wedge(i-1) \wedge\dots\wedge a_k)=-1\quad \mbox{for }a_1\leq i-2,\;i+1\leq a_3<\ldots<a_k, \\
&\qquad \qquad\qquad \qquad \vdots\\
&e_i(a_1\wedge a_2\wedge\dots\wedge i ,\;a_1\wedge a_2\wedge\dots\wedge (i-1))=-1\quad \mbox{for }a_1<a_2<\ldots<a_{k-1}\leq i-2. \\
\end{align*}

\underline{The matrix for $y_j$}.\\
The matrix for $y_j$ is defined recursively by:
\begin{align*}
y_1&=\alpha\Id_{n-1\choose k},  \\
y_{j+1}&=y_j+\diag(s_j)\quad \hbox{ for }1\leq j\leq n-1.
\end{align*}
In closed form, $y_{j+1}=\alpha\Id+\sum_{m=1}^{j}\diag(s_j)$.

We have to check that the matrices defined satisfy the relations of $\sv_n$ in Definition \ref{def svn}. It is clear that the relations involving $y_j$ hold for $\alpha\in\K$ if and only if they hold in the case $\alpha=0$. For the rest of the proof, we therefore consider only the case $\alpha=0$. In the case $\alpha=0$, the representation factors through $\sBr_n$ so we really only need to check the relations defining the subalgebra $\sBr_n$ generated by $e$'s and $s$'s, and that the formula we have declared for the matrix $y_{j+1}$ satisfies $y_{j+1}=s_jy_js_j+s_j+e_j$, the defining relation of the Jucys-Murphy elements. These are relations \ref{def svn}($\vv_1$),($\vv_2$)(i),(ii),($\vv_3$)(i),(ii),($\vv_4$),($\vv_5$)(i),(iii),(v) by Remark \ref{less relns},($\vv_6$), and the repackaging of ($\vv_7$) by Remark \ref{better reln}. By Lemmas \ref{checking Vn a rep of Sn} and \ref{checking it's the reflection rep}, $V_n$ is the standard representation of $S_n$. Then $\Lambda^k V_n$ is an irreducible representation of $S_n$ \cite{FultonHarris}, and we wrote the formula for $s_i$ by computing the action of $s_i$ on $\Lambda^k V_n$ using the action of $s_i$ on $V_n$, so it must hold that $s_i^2=1$, $s_is_{i+1}s_i=s_{i+1}s_is_{i+1}$, and $s_js_i=s_is_j$ if $|i-j|>1$. This puts ($\vv_1$), ($\vv_3$)(i),(ii) to rest. Moreover, $\CC_\alpha(\Lambda^k V_n)$ is irreducible if it is well-defined as an $\sv_n$-representation, since $\Lambda^k V_n$ is an irreducible representation of $S_n$.

If the matrix $e_i$ has entries in a row $r$ then it has $0$'s in column $r$, and vice versa; more precisely, $e_i$ has nonzero entries only in rows involving $i$ in a wedge power, and only in columns involving $i-1$ or $i+1$ in a wedge power but not $i$. It follows immediately that ($\vv_6$) $e_i^2=0$ holds.

Write $s_i=t_i+p_i$ where $t_i=\diag(s_i)$ and $p_i=s_i-t_i$ is the off-diagonal part of $s_i$. Then $p_i=-|e_i|$ where $|e_i|$ is the matrix with entries $|(e_i)_{r,s}|$. As in the proof of Lemma \ref{checking Vn a rep of Sn}, $t_ip_i=-p_i$ and $p_it_i=t_i$, $t_ie_i=-e_i$ and $e_it_i=e_i$, and $p_ie_i=0=e_ip_i$. Then $s_ie_i=t_ie_i=-e_i$ and $e_is_i=e_it_i=e_i$, so $(\vv_5)$(i) holds.

Now we check $(\vv_4)$. First, let us compute the nonzero entries of $e_ie_{i+1}$:
\begin{align*}
&\hbox{Diagonal entries:}\\
&e_ie_{i+1}(a_1\wedge\dots\wedge a_r\wedge i\wedge a_{r+2}\wedge\dots\wedge a_k,\; a_1\wedge\dots\wedge a_r\wedge i\wedge a_{r+2}\wedge\dots\wedge a_k)=-1\\
&\qquad\hbox{ where }a_r\leq i-1\hbox{ and }a_{r+2}\geq i+2. \\
&\hbox{Above-diagonal entries:}\\
&e_ie_{i+1}(a_1\wedge\dots\wedge a_r\wedge i\wedge a_{r+2}\wedge\dots\wedge a_k,\;a_1\wedge \dots\wedge a_r\wedge (i+2)\dots \wedge a_k)=1\\
&\qquad \hbox{ where }a_r\leq i-1\hbox{ and }a_{r+2}\geq i+3.  \\
&\hbox{Below-diagonal entries:}\\
&e_ie_{i+1}(a_1\wedge\dots\wedge a_r\wedge i\wedge (i+1)\wedge\dots\wedge a_k,\;a_1\wedge\dots\wedge a_r\wedge(i-1)\wedge i\wedge\dots\wedge a_k)=1\\
&\qquad\hbox{ where }a_r\leq i-2\hbox{ and }a_{r+3}\geq i+2,  \\
&e_ie_{i+1}(a_1\wedge\dots\wedge a_r\wedge i\wedge(i+1)\wedge\dots\wedge a_k,\; a_1\wedge \dots\wedge a_r\wedge(i-1)\wedge(i+2)\wedge\dots\wedge a_k)=-1\\
&\qquad\hbox{ where }a_r\leq i-2\hbox{ and }a_{r+3}\geq i+3. 
\end{align*}
In order to compute $e_ie_{i+1}e_i$ we write $e_ie_{i+1}$ as the sum of its four types of nonzero entries as above and compute their products with $e_i$ one by one. The diagonal entry $-1$ of $e_ie_{i+1}$ in row $a_1\wedge\dots\wedge a_r\wedge i\wedge a_{r+2}\wedge\dots\wedge a_k$ multiplies the entries of $e_i$ in row $a_1\wedge\dots\wedge a_r\wedge i\wedge a_{r+2}\wedge\dots\wedge a_k$ by $-1$, which gives us all the nonzero entries of $-e_i$ except for those with $a_{r+2}=i+1$ which occur only for the lower diagonal entries; these are obtained from:
\begin{align*}
1=&1\times 1\\
=&e_ie_{i+1}(a_1\wedge\dots\wedge a_r\wedge i\wedge (i+1)\wedge\dots\wedge a_k,\;a_1\wedge\dots\wedge a_r\wedge (i-1)\wedge i\wedge \dots\wedge a_k )\\
&\times e_i( a_1\dots\wedge a_r\wedge (i-1)\wedge i\wedge\dots\wedge a_k,\;a_1\wedge\dots\wedge a_r\wedge(i-1)\wedge(i+1)\wedge\dots\wedge a_k)\\
=&e_ie_{i+1}e_i(a_1\wedge \dots\wedge a_r\wedge i\wedge (i+1)\wedge\dots\wedge a_k,\; a_1\wedge\dots \wedge a_r \wedge(i-1)\wedge(i+1)\wedge\dots\wedge a_k)
\end{align*}
All other products of off-diagonal entries are $0$. This shows $e_ie_{i+1}e_i=-e_i$. 
A similar calculation shows that $e_{i+1}e_ie_{i+1}=-e_{i+1}$.

Next, we check relation $(\vv_5)$(iii). We have to show that $e_ie_{i+1}s_i=-e_is_{i+1}$. As earlier in the proof we compute products with $s_i$ by breaking it into the sum of its diagonal part $t_i$ and off-diagonal part $p_i$. Computing $e_ie_{i+1}s_i$ we have:
\begin{align*}
&e_ie_{i+1}t_i(a_1\wedge\dots\wedge a_r\wedge i\wedge a_{r+2}\wedge\dots\wedge a_k,\; a_1\wedge\dots\wedge a_r\wedge i\wedge a_{r+2}\wedge\dots\wedge a_k)=1\\
&\qquad\hbox{ where }a_r\leq i-1\hbox{ and }a_{r+2}\geq i+2.  \\
&\hbox{Above-diagonal entries:}\\
&e_ie_{i+1}t_i(a_1\wedge\dots\wedge a_r\wedge i\wedge a_{r+2}\wedge\dots\wedge a_k,\;a_1\wedge \dots\wedge a_r\wedge (i+2)\dots \wedge a_k)=1\\
&\qquad \hbox{ where }a_r\leq i-1\hbox{ and }a_{r+2}\geq i+3.  \\
&\hbox{Below-diagonal entries:}\\
&e_ie_{i+1}t_i(a_1\wedge\dots\wedge a_r\wedge i\wedge (i+1)\wedge\dots\wedge a_k,\;a_1\wedge\dots\wedge a_r\wedge(i-1)\wedge i\wedge\dots\wedge a_k)=-1\\
&\qquad\hbox{ where }a_r\leq i-2\hbox{ and }a_{r+3}\geq i+2,  \\
&e_ie_{i+1}t_i(a_1\wedge\dots\wedge a_r\wedge i\wedge(i+1)\wedge\dots\wedge a_k,\; a_1\wedge \dots\wedge a_r\wedge(i-1)\wedge(i+2)\wedge\dots\wedge a_k)=-1\\
&\qquad\hbox{ where }a_r\leq i-2\hbox{ and }a_{r+3}\geq i+3. 
\end{align*}
and by a computation similar to that used in the check of $e_ie_{i+1}e_i=-e_i$ we have
\begin{align*}
&e_ie_{i+1}p_i(a_1\wedge\dots\wedge a_r\wedge i\wedge\dots\wedge a_k,\;a_1\wedge \dots\wedge a_r\wedge (i+1)\wedge\dots\wedge a_k)=1, \\
&e_ie_{i+1}p_i(a_1\wedge\dots\wedge a_r\wedge i\wedge a_{r+2}\wedge\dots\wedge a_k,\;a_1\wedge \dots\wedge a_r\wedge (i-1)\wedge a_{r+2}\wedge\dots\wedge a_k)=1 \\&\qquad \hbox{ if }a_r\leq i-1\hbox{ and }a_{r+2}\geq i+2, 
\end{align*}
whereas
$$e_ie_{i+1}p_i(a_1\wedge\dots\wedge a_r\wedge i\wedge (i+1)\wedge\dots\wedge a_k,\;a_1\wedge\dots\wedge a_r\wedge (i-1)\wedge(i+1)\wedge\dots\wedge a_k)=-1.$$
Computing $-e_is_{i+1}=-e_ip_{i+1}-e_it_{i+1}=e_i|e_{i+1}|-e_it_{i+1}$, we find that $e_i|e_{i+1}|=e_ie_{i+1}t_i$; while $-e_it_{i+1}$ has entries in the same places as in $e_i$ with $1$'s except in those below-diagonal entries where the wedge labeling the column contains $i+1$, and thus we obtain that $e_ie_{i+1}p_i=-e_it_{i+1}$. It follows that $e_ie_{i+1}s_i=-e_is_{i+1}$. A similar calculation verifies $(\vv_5)$(v).

Next, we check the relations in $(\vv_2)$(i),(ii). Since $\Lambda^k V_n$ is an $S_n$-representation, the relations $s_is_j=s_js_i$ if $|i-j|>1$ are automatically satisfied. The calculation in the proof of Lemma \ref{cupcap ideal 1}(\ref{symm subrep}) shows that $e_{i+1}=s_is_{i+1}e_is_{i+1}s_i$ for all $i=1,\dots,n-2$. Thus for (i) it is enough to check
that $e_1s_j=s_je_1$ for all $j\geq 3$, and for (ii) it is then enough to check that $e_1e_3=e_3e_1$.

First, we check $e_1s_j=s_je_1$ for all $j\geq 3$.  The nonzero entries of $e_1$ are given by:
$$e_1(1\wedge a_2\wedge\dots\wedge a_k,\;2\wedge a_2\wedge\dots\wedge a_k)=1\hbox{ for all }3\leq a_3<\ldots<a_k. $$
The nonzero entries in the product $e_1s_j$ are given by:
\begin{align*}
-1&=1(-1)\\
   &=e_1(1\wedge a_2\wedge\dots \wedge a_r\wedge j\wedge a_{r+2}\wedge\dots\wedge a_k,\;2\wedge  a_2\wedge\dots \wedge a_r\wedge j\wedge a_{r+2}\wedge\dots\wedge a_k)\\
&\quad\quad \times s_j(2\wedge  a_2\wedge\dots \wedge a_r\wedge j\wedge a_{r+2}\wedge\dots\wedge a_k,\; 2\wedge  a_2\wedge\dots \wedge a_r\wedge (j+1)\wedge a_{r+2}\wedge\dots\wedge a_k)\\
&=e_1s_j(1\wedge a_2\wedge\dots \wedge a_r\wedge j\wedge a_{r+2}\wedge\dots\wedge a_k,\;2\wedge  a_2\wedge\dots \wedge a_r\wedge (j+1)\wedge a_{r+2}\wedge\dots\wedge a_k)\\
&\qquad\qquad\hbox{ for all }j\geq 3,\; a_r\leq j-1,\;a_{r+2}\geq j+2,  \\
-1&=1(-1) \\ 
&=e_1(1\wedge a_2\wedge\dots \wedge a_r\wedge j\wedge a_{r+2}\wedge\dots\wedge a_k,\;2\wedge  a_2\wedge\dots \wedge a_r\wedge j\wedge a_{r+2}\wedge\dots\wedge a_k)\\
&\quad\quad \times s_j(2\wedge  a_2\wedge\dots \wedge a_r\wedge j\wedge a_{r+2}\wedge\dots\wedge a_k,\; 2\wedge  a_2\wedge\dots \wedge a_r\wedge (j-1)\wedge a_{r+2}\wedge\dots\wedge a_k)\\
&=e_1s_j(1\wedge a_2\wedge\dots \wedge a_r\wedge j\wedge a_{r+2}\wedge\dots\wedge a_k,\;2\wedge  a_2\wedge\dots \wedge a_r\wedge (j-1)\wedge a_{r+2}\wedge\dots\wedge a_k)\\
&\qquad\qquad\hbox{ for all }j\geq 4,\; a_r\leq j-2,\;a_{r+2}\geq j+1,  \\
e_1s_j&(1\wedge a_2\wedge\dots\wedge a_k,\;2\wedge a_2\wedge \dots\wedge a_k) \\
&=e_1(1\wedge a_2\wedge\dots\wedge a_k,\;2\wedge a_2\wedge \dots\wedge a_k)s_j(2\wedge a_2\wedge \dots\wedge a_k,\;2\wedge a_2\wedge \dots\wedge a_k)\\
&=\begin{cases} 
-1&\hbox{ if }j\in\{a_2,\dots,a_k\}, \\
\:\:\:1&\hbox{ else}. 
\end{cases}
\end{align*}
The nonzero entries in the product $s_je_1$ are given by:
\begin{align*}
-1&=(-1)1 \\ 
&=s_j(1\wedge  a_2\wedge\dots \wedge a_r\wedge j\wedge a_{r+2}\wedge\dots\wedge a_k,\; 1\wedge  a_2\wedge\dots \wedge a_r\wedge (j+1)\wedge a_{r+2}\wedge\dots\wedge a_k)\\
&\quad\quad \times e_1(1\wedge a_2\wedge\dots \wedge a_r\wedge j\wedge a_{r+2}\wedge\dots\wedge a_k,\;2\wedge  a_2\wedge\dots \wedge a_r\wedge j\wedge a_{r+2}\wedge\dots\wedge a_k)\\
&=s_je_1(1\wedge a_2\wedge\dots \wedge a_r\wedge j\wedge a_{r+2}\wedge\dots\wedge a_k,\;2\wedge  a_2\wedge\dots \wedge a_r\wedge (j+1)\wedge a_{r+2}\wedge\dots\wedge a_k)\\
&\qquad\qquad\hbox{ for all }j\geq 3,\; a_r\leq j-1,\;a_{r+2}\geq j+2,  \\
-1&=(-1)1 \\
&=s_j(1\wedge  a_2\wedge\dots \wedge a_r\wedge j\wedge a_{r+2}\wedge\dots\wedge a_k,\; 1\wedge  a_2\wedge\dots \wedge a_r\wedge (j-1)\wedge a_{r+2}\wedge\dots\wedge a_k)\\
&\quad\quad \times e_1(1\wedge a_2\wedge\dots \wedge a_r\wedge j\wedge a_{r+2}\wedge\dots\wedge a_k,\;2\wedge  a_2\wedge\dots \wedge a_r\wedge j\wedge a_{r+2}\wedge\dots\wedge a_k)\\
&=s_je_1(1\wedge a_2\wedge\dots \wedge a_r\wedge j\wedge a_{r+2}\wedge\dots\wedge a_k,\;2\wedge  a_2\wedge\dots \wedge a_r\wedge (j-1)\wedge a_{r+2}\wedge\dots\wedge a_k)\\
&\qquad\qquad\hbox{ for all }j\geq 4,\; a_r\leq j-2,\;a_{r+2}\geq j+1, \\
s_je_1&(1\wedge a_2\wedge\dots\wedge a_k,\;2\wedge a_2\wedge \dots\wedge a_k) \\
&=s_j(1\wedge a_2\wedge \dots\wedge a_k,\;1\wedge a_2\wedge \dots\wedge a_k)
e_1(1\wedge a_2\wedge\dots\wedge a_k,\;2\wedge a_2\wedge \dots\wedge a_k)\\
&=\begin{cases} 
-1&\hbox{ if }j\in\{a_2,\dots,a_k\}, \\
\:\:\:1&\hbox{ else}. 
\end{cases}
\end{align*}
So $e_1s_j=s_je_1$ for any $j\geq 3$. By the remarks above, this implies $(\vv_2)$(i) holds.

Let's check that $e_1e_3=e_3e_1$. The nonzero entries of $e_3$ whose row entry begins with a $2$ or whose column entry begins with a $1$ are given by:
\begin{alignat*}{2}
&e_3(2\wedge 3\wedge a_3\wedge\dots\wedge a_k,\;2\wedge 4\wedge a_3\wedge\dots\wedge a_k)=1\quad &&\hbox{ for all }5\leq a_3<\ldots<a_k, \\
&e_3(1\wedge 3\wedge a_3\wedge\dots\wedge a_k,\;1\wedge 4\wedge a_3\wedge\dots\wedge a_k)=1\quad &&\hbox{ for all }5\leq a_3<\ldots<a_k, \\
&e_3(1\wedge 2\wedge 3\wedge a_4\wedge\dots\wedge a_k,\;1\wedge 2\wedge 4\wedge\dots\wedge a_k)=1\quad &&\hbox{ for all }5\leq a_4<\ldots<a_k, \\
&e_3(1\wedge 3\wedge\dots\wedge a_k,\;1\wedge 2\wedge\dots\wedge a_k)=-1\quad &&\hbox{ for all }4\leq a_2<\ldots<a_k. 
\end{alignat*}
However, because $a_2\geq 3 $ in the nonzero entries of $e_1$, the last two types of term for $e_3$ never match up to with the rows or columns of the nonzero entries of $e_1$ since $e_1$ has no entries in rows or columns labeled $1\wedge 2\wedge\cdots$. Computing the nonzero entries of $e_1e_3$ and $e_3e_1$ we get for all $5\leq a_3<\ldots<a_k$:
\begin{align*}
1&=e_1(1\wedge3\wedge a_3\wedge\dots\wedge a_k,\;2\wedge 3\wedge a_3\wedge\dots\wedge a_k)e_3(2\wedge 3\wedge a_3\wedge \dots\wedge a_k,\;2\wedge 4\wedge a_3\wedge\dots\wedge a_k) \\
&=e_1e_3(1\wedge3\wedge a_3\wedge\dots\wedge a_k,\;2\wedge 4\wedge a_3\wedge\dots\wedge a_k), \\
1&=e_3(1\wedge 3\wedge a_3\wedge \dots\wedge a_k,\;1\wedge 4\wedge a_3\wedge \dots \wedge a_k)e_1(1\wedge 4\wedge a_3\wedge \dots\wedge a_k,\;2\wedge 4\wedge a_3\wedge \dots\wedge a_k)\\
&=e_3e_1(1\wedge 3\wedge a_3\wedge \dots\wedge a_k,\;2\wedge 4\wedge a_3\wedge\dots\wedge a_k). 
\end{align*}
So $e_1e_3=e_3e_1$. By the remarks above, this implies $(\vv_2)$(ii) holds.

The last thing to check is $(\vv_7)$, as the single relation $y_{j+1}=s_jy_js_j+s_j+e_j$. Break $s_j$ into the sum of its diagonal and off-diagonal parts: $s_j=t_j+p_j$. By a check similar to the check that $s_j$ and $e_1$ commute, $t_i$ and $p_j$ commute for all $1\leq i\leq j-2$. Obviously $t_j$ and $t_i$ commute since they are diagonal matrices. So $t_i$ and $s_j$ commute for all $i\leq j-2$. We then have
$$s_jy_js_j+s_j+e_j=s_j\left(\sum_{i=1}^{j-1} t_i\right)s_j+s_j+e_j=\left(\sum_{i=1}^{j-2} t_i\right)+s_jt_{j-1}s_j+s_j+e_j, $$
and so the relation $y_{j+1}=s_jy_js_j+s_j+e_j$ holds if and only if $s_jt_{j-1}s_j+s_j+e_j=t_{j-1}+t_j$. Let's check this. The diagonal matrix $t_{j-1}t_j=t_jt_{j-1}$ has entries given by:
$$t_{j-1}t_j(a_1\wedge\dots\wedge a_k,a_1\wedge\dots\wedge a_k)
=\begin{cases}-1 &\hbox{ if }j-1\hbox{ or }j\in\{a_1,\dots,a_k\}\hbox{ but not both},\\ 
\:\:\:1 &\hbox{ else}.
\end{cases}\\$$
We have
$$s_jt_{j-1}s_j=(t_j+p_j)t_{j-1}(t_j+p_j)=t_{j-1}+p_j(t_{j-1}t_j)+(t_{j-1}t_j)p_j+p_jt_{j-1}p_j,$$
and an argument similar to that used to verify $e_j^2=0$ also shows $p_jt_{j-1}p_j=0$.
We compute:
\begin{align*}
&p_jt_{j-1}t_j(a_1\wedge\dots\wedge a_r\wedge j\wedge\dots\wedge a_k,\;a_1\wedge\dots\wedge a_r\wedge (j+1)\wedge\dots\wedge a_k)\\
&\qquad\qquad\qquad\qquad\qquad=
\begin{cases}(-1)(-1)=1&\hbox{ if }a_r=j-1, \\
(-1)(1)=-1&\hbox{ if }a_r\neq j-1, 
\end{cases}\\
&p_jt_{j-1}t_j(a_1\wedge\dots\wedge a_r\wedge j\wedge\dots\wedge a_k,\;a_1\wedge\dots\wedge a_r\wedge (j-1)\wedge\dots\wedge a_k)=(-1)(-1)=1, 
\end{align*}
where there is no case distinction for the lower-diagonal term since in that case, $a_{r+2}>j$. On the other hand,
\begin{align*}
&t_{j-1}t_jp_j(a_1\wedge\dots\wedge a_r\wedge j\wedge\dots\wedge a_k,\;a_1\wedge\dots\wedge a_r\wedge(j+1)\wedge\dots\wedge a_k)=\\
&\qquad\qquad\qquad\qquad\qquad=
\begin{cases}(1)(-1)=-1&\hbox{ if }a_r=j-1,\\
(-1)(-1)=1&\hbox{ if }a_r\neq j-1,
\end{cases}\\
&t_{j-1}t_jp_j(a_1\wedge\dots\wedge a_r\wedge j\wedge\dots\wedge a_k,\;a_1\wedge\dots\wedge a_r\wedge (j-1)\wedge\dots\wedge a_k)=(-1)(-1)=1, 
\end{align*}
where there is no case distinction for the lower-diagonal term, this time because\\ $a_r<j-1$.
The only nonzero terms of $t_{j-1}t_{j}p_j+p_jt_{j-1}t_{j}$ are then:
\begin{align*}
(t_{j-1}t_{j}p_j+p_jt_{j-1}t_{j})&(a_1\wedge\dots\wedge a_r\wedge j\wedge\dots\wedge a_k,\;a_1\wedge\dots\wedge a_r\wedge (j-1)\wedge\dots\wedge a_k)=2\\
&\hbox{ for all }a_1<\ldots<a_r\leq j-2,\;j+1\leq a_r<\ldots<a_k, 
\end{align*}
and on the other hand, the only nonzero entries of $p_j+e_j$ are:
\begin{align*}
(p_j+e_j)&(a_1\wedge\dots\wedge a_r\wedge j\wedge\dots\wedge a_k,\;a_1\wedge\dots\wedge a_r\wedge (j-1)\wedge\dots\wedge a_k)=-2\\
&\hbox{ for all }a_1<\ldots<a_r\leq j-2,\;j+1\leq a_r<\ldots<a_k. 
\end{align*}
Therefore, we have:
$$s_jt_{j-1}s_j+s_j+e_j=t_{j-1}+p_jt_{j-1}t_j+t_{j-1}t_jp_j+t_j+p_j+e_j=t_{j-1}+t_j$$
and so 
$$y_{j+1}=\sum_{i=0}^{j} t_i=s_jy_js_j+s_j+e_j$$
as desired.

This finishes the check of the relations giving a calibrated representation of $\sBr_n$, and thus of a calibrated $\sv_n$-representation on which $y_1$ acts by a scalar $\alpha\in\K$.
\end{proof}

\begin{corollary}\label{JM evalues} The eigenvalues of $Y_j$ acting on $\Lambda^k V_n$ lie in $\{-k,-k+1,-k+2,\dots,n-k-2,n-k-1 \}$. Moreover, $Y_n$ acts by the scalar matrix $(n-1-2k)\Id_{n-1\choose k}$.
\end{corollary}
\begin{proof} The $a$'th diagonal entry of $y_j$ is the sum of the $a$'th diagonal entries of $s_i$ for $1\leq i\leq j$. The diagonal entries of $s_i$ are $1$ and $-1$. For a given $s_i$, the diagonal entry corresponding to a wedge $a_1\wedge a_2\wedge \dots \wedge a_k$ is $-1$ if $a_r=i$ for some $r=1,\dots, k$ and $1$ otherwise. The number of times $-1$ occurs in the $a$'th position on the diagonals of all the $s_i$'s is therefore equal to $k$, the number of times $1$ occurs in the $a$'th position on the diagonals of all the $s_i$'s is equal to $n-k-1$. This gives the bounds on the eigenvalues, and since $y_n=\sum_{i=1}^{n-1}\diag(s_i)$ it follows that $y_n$ is a scalar matrix with scalar the total number of $1$'s minus the total number of $-1$'s which is $(n-k-1)-k$.
\end{proof}

\begin{example}
Let us continue with the example of $V_n$ when $n=5$. We saw that the standard representation $V_5$ of $S_5$ extends to a calibrated representation of $\sv_5$ with nonzero action of $e_i$ for $1\leq i\leq 4$. Let $\overline{v}_1$, $\overline{v}_2$, $\overline{v}_3$, $\overline{v}_4$ be the basis of simultaneous eigenvectors for $y_j$, $1\leq j\leq 5$, with respect to which we wrote the matrices in Theorem \ref{reflection rep}. By Theorem \ref{exterior powers},  $\Lambda^2 V_5$ and $\Lambda^3 V_5$ also extend to irreducible calibrated representations of $\sv_5$ with nonzero action of $e_i$. Here we give the explicit matrices for the actions of the generators of $\sv_5$ by matrices on these two representations.
\begin{enumerate}
\item 
Consider $\Lambda^2 V_5$ with basis $\cB_2=\{\overline{v}_{1}\wedge \overline{v}_{2},\;\overline{v}_{1}\wedge \overline{v}_{3},\;\overline{v}_{1}\wedge \overline{v}_{4},\;\overline{v}_{2}\wedge \overline{v}_{3},\;\overline{v}_{2}\wedge \overline{v}_{4},\;\overline{v}_{3}\wedge \overline{v}_{4}\}$. The matrices for $s_i$ are given by
\begin{align*}
&s_1=\small 
\begin{pmatrix}-1&0&0&0&0&0\\
0&-1&0&-1&0&0\\
0&0&-1&0&-1&0\\
0&0&0&1&0&0\\
0&0&0&0&1&0\\
0&0&0&0&0&1
\end{pmatrix},
\quad
s_2=\small 
\begin{pmatrix}
-1&-1&0&0&0&0\\
0&1&0&0&0&0\\
0&0&1&0&0&0\\
0&-1&0&-1&0&0\\
0&0&-1&0&-1&-1\\
0&0&0&0&0&1
\end{pmatrix},\\
&s_3=\small 
\begin{pmatrix}
1&0&0&0&0&0\\
-1&-1&-1&0&0&0\\
0&0&1&0&0&0\\
0&0&0&-1&-1&0\\
0&0&0&0&1&0\\
0&0&0&0&-1&-1
\end{pmatrix},
\quad
s_4=\small 
\begin{pmatrix}
1&0&0&0&0&0\\
0&1&0&0&0&0\\
0&-1&-1&0&0&0\\
0&0&0&1&0&0\\
0&0&0&-1&-1&0\\
0&0&0&0&0&-1
\end{pmatrix}. 
\end{align*}

The matrices for $e_i$ are given by:
\begin{align*}
&e_1=\small 
\begin{pmatrix}0&0&0&0&0&0\\
0&0&0&1&0&0\\
0&0&0&0&1&0\\
0&0&0&0&0&0\\
0&0&0&0&0&0\\
0&0&0&0&0&0
\end{pmatrix},
\quad
e_2=\small 
\begin{pmatrix}
0&1&0&0&0&0\\
0&0&0&0&0&0\\
0&0&0&0&0&0\\
0&-1&0&0&0&0\\
0&0&-1&0&0&1\\
0&0&0&0&0&0
\end{pmatrix},\\
&e_3=\small 
\begin{pmatrix}
0&0&0&0&0&0\\
-1&0&1&0&0&0\\
0&0&0&0&0&0\\
0&0&0&0&1&0\\
0&0&0&0&0&0\\
0&0&0&0&-1&0
\end{pmatrix},
\quad
e_4=\small 
\begin{pmatrix}
0&0&0&0&0&0\\
0&0&0&0&0&0\\
0&-1&0&0&0&0\\
0&0&0&0&0&0\\
0&0&0&-1&0&0\\
0&0&0&0&0&0
\end{pmatrix}. 
\end{align*}
The matrices for $y_j$ are given by:
\begin{align*}
& y_1=0,
\quad y_2=\diag(-1,-1,-1,1,1,1),\quad y_3=\diag(-2,0,0,0,0,2), \\
&y_4=\diag(-1,-1,1,-1,1,1),\quad y_5=0. 
\end{align*}

Observe that $e_i$ does not act by $0$ and that $e_ie_j$ also does not act by 0. This means that the irreducible is of the form $\LL(\lambda)$ with $\lambda$ a partition of $n-4=2$ \cite{Coulembier}.

\item Consider $\Lambda^3 V_5$ with basis $\cB_3=\{\overline{v}_1\wedge\overline{v}_2\wedge\overline{v}_3,\;\overline{v}_1\wedge\overline{v}_2\wedge\overline{v}_4,\;\overline{v}_1\wedge\overline{v}_3\wedge\overline{v}_4,\;\overline{v}_2\wedge\overline{v}_3\wedge\overline{v}_4\}$. The matrices for $s_i$ are given by:
\begin{align*}
&s_1=\small 
\begin{pmatrix}
-1&0&0&0\\0&-1&0&0\\0&0&-1&-1\\0&0&0&1
\end{pmatrix},\quad
s_2=\small 
\begin{pmatrix}
-1&0&0&0\\0&-1&-1&0\\0&0&1&0\\0&0&-1&-1
\end{pmatrix},\\
&s_3=\small 
\begin{pmatrix}
-1&-1&0&0\\0&1&0&0\\0&-1&-1&0\\0&0&0&-1
\end{pmatrix},\quad
s_4=\small 
\begin{pmatrix}
1&0&0&0\\-1&-1&0&0\\0&0&-1&0\\0&0&0&-1
\end{pmatrix}. 
\end{align*}
The matrices for $e_i$ are given by:
\begin{align*}
&e_1=\small 
\begin{pmatrix}
0&0&0&0\\0&0&0&0\\0&0&0&1\\0&0&0&0
\end{pmatrix},\quad
e_2=\small 
\begin{pmatrix}
0&0&0&0\\0&0&1&0\\0&0&0&0\\0&0&-1&0
\end{pmatrix},\\
&e_3=\small 
\begin{pmatrix}
0&1&0&0\\0&0&0&0\\0&-1&0&0\\0&0&0&0
\end{pmatrix},\quad
e_4=\small 
\begin{pmatrix}
0&0&0&0\\-1&0&0&0\\0&0&0&0\\0&0&0&0
\end{pmatrix}. 
\end{align*}
The matrices for $y_j$ are given by:
\begin{align*}&y_1=0,\quad y_2=\diag(-1,-1,-1,1),\quad y_3=\diag(-2,-2,0,0),\\
&\quad y_4=\diag(-3,-1,-1,-1),\quad y_5=\diag(-2,-2,-2,-2). \end{align*}
Observe that $e_i$ does not act by $0$ but $e_ie_j$ does for $|j-i|>1$. This means that the irreducible is of the form $\LL(\lambda)$ with $\lambda$ a partition of $n-2=4$ \cite{Coulembier}.
\end{enumerate}

\end{example}


\section{Pascal's triangle and a Bratteli diagram} \label{branching}
In this section we continue to follow the Okounkov--Vershik approach \cite{OkounkovVershik} in order to describe the spectrum and branching of the irreducible calibrated $\sBr_n$-representations $\CC_0(\Lambda^k V_n)$ constructed in the previous section. For similar techniques applied to diagram algebras, see for instance the representation theory of the partition algebra as treated by \cite{Enyang1},\cite{Enyang2},\cite{Bowmanetal}; see \cite{BowmanEnyangGoodman} for an axiomatic approach to towers of diagram algebras admitting a cellular basis.
\begin{definition} If $\LL$ is a calibrated representation, its \textit{spectrum} $\Spec\LL$ is the set of all $(\alpha_1,\dots,\alpha_n)\in\K^n$ such that there is a  $\overline{v}\in\LL$ with $Y_j\overline{v}=\alpha_j\overline{v}$ for $j=1,\dots,n$. 
\end{definition}
An element of $\Spec\LL$ corresponding to an eigenvector $\overline{v}$ is exactly the sequence of the $(a,a)$'th matrix entries of $y_1,\dots,y_n$ if $\overline{v}$ is the $a$'th basis vector for $\LL$. By considering the branching graph of the calibrated representations $\CC_0(\Lambda^k V_n)$, we will see how to easily describe $\Spec\CC_0(\Lambda^k V_n)$ in a completely explicit, combinatorial way.

\subsection{Restriction of $\CC_0(\Lambda^k V_n)$ from $\sBr_n$ to $\sBr_{n-1}$}
The underlying $\K S_n$-representation of $\CC_0(\Lambda^k V_n)$ is $\Lambda^k V_n$, which restricts from $S_n$ to $S_{n-1}$ as: $$\Res^{S_n}_{S_{n-1}}\Lambda^k V_n=\Lambda^{k-1} V_{n-1}\oplus \Lambda^k V_{n-1}.$$
We would like to know how the irreducible calibrated $\sBr_n$-representation $\CC_0(\Lambda^k V_n)$ decomposes into irreducible $\sBr_{n-1}$ representations when restricted to the copy of $\sBr_{n-1}\subset \sBr_{n}$ generated by $e_1,\dots, e_{n-2},s_1,\dots,s_{n-2}$ and whether the restricted representation is semisimple.

\begin{theorem}\label{res special}
For $1\leq k\leq n-2$, the restriction $\Res^{\sBr_n}_{\sBr_{n-1}}\CC_0(\Lambda^k V_n)$ of the irreducible calibrated $\sBr_n$-representation $\CC_0(\Lambda^k V_n)$ is indecomposable, and there is a non-split short exact sequence of calibrated $\sBr_{n-1}$-representations
$$0\rightarrow \CC_0(\Lambda^{k} V_{n-1})\rightarrow \Res^{\sBr_n}_{\sBr_{n-1}}\CC_0(\Lambda^k V_n)\rightarrow \CC_0(\Lambda^{k-1} V_{n-1})\rightarrow 0. $$
For the extremal cases $k=0$ and $k=n-1$ we have $\Res^{\sBr_n}_{\sBr_{n-1}}\CC_0(\Lambda^0 V_n)=\CC_0(\Lambda^0 V_{n-1})$ and $\Res^{\sBr_n}_{\sBr_{n-1}}\CC_0(\Lambda^{n-1} V_n)=\CC_0(\Lambda^{n-2} V_{n-1})$.
\end{theorem}

\begin{proof}
The extremal cases of $\Lambda^0 V_n$ and $\Lambda^{n-1}V_n$ are obvious since these are the trivial and the sign representations of $S_n$, respectively. For the rest of the proof we assume that $1\leq k\leq n-2$. 
Any $\sBr_n$-representation $V$ has a unique decomposition as a direct sum of irreducible $\K S_n$-representations, and if $V$ is irreducible as a $\K S_n$-representation then obviously $V$ is irreducible as an $\sBr_n$-representation. So there are two possibilities: either (i) $\Res^{\sBr_n}_{\sBr_{n-1}}\CC_0(\Lambda^k V_n)$ is indecomposable with two distinct composition factors isomorphic to $\Lambda^{k-1} V_{n-1}$ and $ \Lambda^k V_{n-1}$ as $\K S_{n-1}$-representations; or (ii) $\Res^{\sBr_n}_{\sBr_{n-1}}\CC_0(\Lambda^k V_n)$ is the direct sum of two irreducible $\sBr_{n-1}$-representations, isomorphic to $\Lambda^{k-1} V_{n-1}$ and $ \Lambda^k V_{n-1}$ as $\K S_{n-1}$-representations. In case (i), we have to further determine whether $\Res^{\sBr_n}_{\sBr_{n-1}}\CC_0(\Lambda^k V_n)$ is irreducible or whether it has two irreducible composition factors.

First, we establish that case (i) holds: $\Res^{\sBr_n}_{\sBr_{n-1}}\CC_0(\Lambda^k V_n)$ is indecomposable for any $k=0,\dots,n-1$. To check the cases $1\leq k\leq n-2$, we compute $\End_{\sBr_{n-1}}(\Res^{\sBr_n}_{\sBr_{n-1}}\CC_0(\Lambda^k V_n))$ and find that it is a local ring. 
First, we claim that $\End_{\sBr_{n-1}}(\Res^{\sBr_n}_{\sBr_{n-1}}\CC_0(\Lambda^k V_n))$ is contained in diagonal matrices. Let us compute the commutant of the matrices for $Y_2,\dots,Y_{n-1}$: write $\mathrm{Comm}(Y_2,\dots,Y_{n-1})$ for the subalgebra of ${n-1\choose k}\times {n-1\choose k}$ matrices commuting with $Y_2,\dots,Y_{n-1}$. Obviously, $\End_{\sBr_{n-1}}(\Res^{\sBr_n}_{\sBr_{n-1}}\CC_0(\Lambda^k V_n))\subseteq\mathrm{Comm}(Y_2,\dots,Y_{n-1})$. Since not only $Y_1$ but also $Y_n$ is a multiple of the identity matrix by Corollary \ref{JM evalues}, we may throw in $Y_1$ and $Y_n$ at no cost: $\mathrm{Comm}(Y_2,\dots,Y_{n-1})=\mathrm{Comm}(Y_1,Y_2,\dots,Y_{n-1},Y_n)$. Let $X=(x_{\ell,m})\in\mathrm{Comm}(Y_1,Y_2,\dots,Y_{n-1},Y_n)$. Now, suppose that for some $\ell\neq m$, $\ell,m\in\{1,\dots,{n-1\choose k}\}$, and some $i\in\{1,\dots,n-1\}$, it holds that $(Y_i)_{\ell,\ell}-(Y_{i+1})_{\ell,\ell}=-1$ and $(Y_i)_{m,m}-(Y_{i+1})_{m,m}=1$. Computing the $\ell,m$'th  matrix entry of $X(Y_i-Y_{i+1})$ we get that it is minus the $\ell,m$'th matrix entry of $(Y_i-Y_{i+1})X$, but $X(Y_i-Y_{i+1})=(Y_i-Y_{i+1})X$, and therefore $x_{\ell,m}=0$. By the same argument, $x_{m,\ell}=0$. However, by the definition of the action in Theorem \ref{exterior powers} we have $(Y_i)_{\ell,\ell}-(Y_{i+1})_{\ell,\ell}=-1$ whenever $(s_i)_{\ell,\ell}=-1$ and $(Y_i)_{m,m}-(Y_{i+1})_{m,m}=1$ whenever $(s_i)_{m,m}=1$ for all $i=1,\dots,n-1$. It is clear from the definition of the matrix of $s_i$ in Theorem \ref{exterior powers} that for any $\ell\neq m$ there is some $i$ such that $(s_i)_{\ell,\ell}=-(s_i)_{m,m}$. Thus all off-diagonal entries of $X$ vanish: $x_{\ell,m}=0=x_{m,\ell}$ for all $\ell\neq m$, and so $X$ is a diagonal matrix. Therefore $\End_{\sBr_{n-1}}(\Res^{\sBr_n}_{\sBr_{n-1}}\CC_0(\Lambda^k V_n))$ is contained in diagonal matrices.

Next, we claim that the diagonal matrices commuting with $e_1,\dots,e_{n-2}$ are just the scalar matrices. Observe that $i=1,\dots,n-2$ are all of the $i$ such that $e_i$ has nonzero entries above the diagonal, and that all nonzero entries above the diagonal are $1$'s. Moreover, if $(e_i)_{\ell,m}=1$ then $(e_j)_{\ell,m}=0$ for $i\neq j$, where $\ell<m$ are wedges of $k$ integers from $\{1,\dots,n-1\}$ with the wedges ordered lexicographically, i.e. the first row of the matrix is labeled $1\wedge 2\wedge\dots\wedge k$ and the last is labeled $(n-k)\wedge (n-k+1)\wedge\dots\wedge (n-1)$. If we take the sum $e:=\sum_{i=1}^{n-2}e_i$ then the entries above the diagonal of $e$ are $0$'s and $1$'s. The position of the $1$'s is defined in terms of increasing some $a_m$ to $a_m+1$ in the wedge $a_1\wedge a_2 \wedge\dots\wedge a_k$, where $a_{m+1}>a_m+1$. Starting from the first row $\ell=1\wedge 2\wedge\dots \wedge k$ and then successively increasing one of the numbers in the wedge by one at a time (where $a$ can only be increased to $a+1$ if $a+1$ does not already appear in the wedge), one traverses every $a_1\wedge\dots\wedge a_k$ in a row or column. It follows  that the only diagonal matrices commuting with $e$ are the scalar matrices. Therefore $\End_{\sBr_{n-1}}(\Res^{\sBr_n}_{\sBr_{n-1}}\CC_0(\Lambda^k V_n))=\K.$ This proves $\Res^{\sBr_n}_{\sBr_{n-1}}\CC_0(\Lambda^k V_n)$ is indecomposable. 

Now that we know $V:=\Res^{\sBr_n}_{\sBr_{n-1}}\CC_0(\Lambda^k V_n)$ is indecomposable, we have to determine whether it is irreducible or an extension of two irreducible representations. 
Let
\begin{align*}
I&=\left\{i\in\left\{1,2,\dots,{n-1\choose k}\right\} \mid  (Y_{n-1}-Y_n)\overline{v}_i=-\overline{v}_i\right\}, \\
J&=\left\{j\in\left\{1,2,\dots,{n-1\choose k}\right\}\mid (Y_{n-1}-Y_n)\overline{v}_j=\overline{v}_j\right\}. 
\end{align*}
Then $\{1,2,\dots,{n-1\choose k}\}=I\sqcup J$. We claim that $W:=\{\overline{v}_i\mid i\in I\}$ is an $\sBr_n$-subrepresentation of $V$ isomorphic to $\CC_0(\Lambda^{k}V_{n-1})$ and that $V/W\cong \CC_0( \Lambda^{k-1} V_{n-1})$ with basis given by the images of $\overline{v}_j$ for $j\in J$. First, observe that these vector subspaces indeed have the desired dimensions: 
\begin{align*}
|I|&=\#\left\{i\in\left\{1,2,\dots,{n-1\choose k}\right\}\mid (Y_{n-1}-Y_n)_{i,i}=-1\right\}\\
&=\#\left\{i\in\left\{1,2,\dots,{n-1\choose k}\right\}\mid (s_{n-1})_{i,i}=1\right\}\\
&=\#\{a_1\wedge a_2\wedge\dots\wedge a_k: 1\leq a_1<a_2<\ldots<a_k< n-1\}\\
&={n-2\choose k}=\dim\Lambda^{k}V_{n-1},
\end{align*}
and thus $$|J|={n-1\choose k}-{n-2\choose k}={n-2\choose k-1}=\dim \Lambda^{k-1} V_{n-1}.$$

Next, it is clear from the definition of the matrix $s_i$ that $s_1,\dots,s_{n-2}$ preserve the subspace $W$. Indeed: let $1\leq i\leq n-2$ and consider a wedge $(a_1\wedge\dots\wedge i\wedge\dots\wedge n-1)$ labeling a row of $s_i$ with off-diagonal nonzero entries and labeling a basis vector of the vector space complement to $W$ in $V$. Then $s_i$ has off-diagonal entries in the columns $(a_1\wedge\dots\wedge i\pm 1\wedge\dots\wedge a_{k-1}\wedge n-1)$ of this row (so long this does not cause a repeated entry making the wedge equal to $0$.) In either case, the column of the nonzero entry contains an $n-1$ in it so does not label one of the basis vectors of $W$. Thus if $\overline{v}_k\in W$, $k\in I$, is one of the eigenvectors in $W$ then $s_i\cdot \overline{v}_k\cap \mathrm{Span}\{\overline{v}_j\mid j\in J\}=\emptyset$. Moreover, $W$ must be irreducible as an $S_n$-representation since $V$ has exactly two composition factors as an $S_n$-representation and $0\subsetneq W\subsetneq V$.

Moreover, it is clear that $e_1$ preserves $W$ since the only nonzero entries of $e_1$ are given by
$e_1(1\wedge a_2\wedge\dots\wedge a_k,\;2\wedge a_2\wedge\dots\wedge a_k)=1$. Then by Lemma \ref{cupcap ideal 1}(\ref{symm subrep}) it follows that $W$ is an irreducible $\sBr_{n-1}$-subrepresentation of $V$. As $W\subsetneq V$ and $V$ has at most two composition factors, it follows that $V/W$ is an irreducible $\sBr_n$-representation as well. Now, $W$ is calibrated, the matrices for $s_1,\dots,s_{n-2}$ acting on $V$ are exactly those of $\Lambda^k V_{n-1}$, the matrices of $e_1,\dots,e_{n-2}$ have their nonzero entries in exactly the off-diagonal spots where $s_1,\dots,s_{n-2}$ have nonzero entries with $1$s above the diagonal and $-1$'s below the diagonal, and the eigenvalues of $Y_1,\dots,Y_{n-1}$ on $W$ are then given by Lemma \ref{evalues closed form} and match the formula in Theorem \ref{exterior powers}. So $W\cong\CC_0(\Lambda^k V_{n-1})$. By the same argument, $V/W\cong \CC_0(\Lambda^{k-1}V_{n-1})$. Therefore we have a short exact sequence 
$$0\rightarrow \CC_0(\Lambda^{k} V_{n-1})\rightarrow \Res^{\sBr_n}_{\sBr_{n-1}}\CC_0(\Lambda^k V_n)\rightarrow \CC_0(\Lambda^{k-1} V_{n-1})\rightarrow 0$$
as claimed.
\end{proof}

\subsection{A Bratteli diagram}

Consider the set of nonempty row or column partitions of size congruent to $n$ mod $2$ and at most $n$: $$\mathcal{RC}_n:=\left\{(k),(1^k)\mid k\in\left\{n-2r\mid 0\leq r\leq \Big\lfloor \frac{n-1}{2}\Big\rfloor\right\}\right\}.$$
Label the representations of $\CC_0(\Lambda^k V_n)$, $k=0,\dots, n-1$, by the elements of $\mathcal{RC}_n$ as follows: 
\begin{table}[H]
\centering
\begin{tabular}{c|c|c|c|c|c|c}
$\Lambda^0 V_n$ & $\Lambda^1 V_n$ & $\Lambda^2 V_n$ & $\cdots$ & $\Lambda^{n-3} V_n$ & $\Lambda^{n-2} V_n$ & $\Lambda^{n-1} V_n$\\
\hline
&&&&&&\\
$(n)$ & $(n-2)$ & $(n-4)$ & $\cdots$ & $(1^{n-4})$ & $(1^{n-2})$ & $(1^{n})$
\end{tabular}
\end{table}

Now define a graph $\Gamma$ whose vertices at level $n$ are given by $\mathcal{RC}_n$. There is an arrow from $\lambda$ at level $n$ to $\mu$ at level $n+1$ in $\Gamma$ if and only if $[\Res^{\sBr_{n+1}}_{\sBr_{n}}\CC_0(\Lambda^k V_{n+1}):\CC_0(\Lambda^j V_{n})]=1$ and $\CC_0(\Lambda^j V_n)$ corresponds to $\lambda$, $\CC_0(\Lambda^k V_{n+1})$ to $\mu$. Then $\Gamma$ is the branching graph of the tower of irreducible calibrated representations we have constructed in this section. The graph $\Gamma$ is isomorphic to Pascal's triangle; from level $1$ to level $6$ it looks like this:

$$\TikZ{[scale=.5]
\draw
(0,-2)node[]{$\tiny \young(0)$}
(2,-4)node[]{$\tiny \young(0,\yminusone)$}
(-2,-4)node[]{$\tiny \young(01)$}
(4,-6)node[]{$\tiny\young(0,\yminusone,\yminustwo)$}
(0,-6)node[]{$\tiny\young(0)$}
(-4,-6)node[]{$\tiny\young(012)$}
(6,-8)node[]{$\tiny\young(0,\yminusone,\yminustwo,\yminusthree)$}
(2,-8)node[]{$\tiny\young(0,\yminusone)$}
(-2,-8)node[]{$\tiny\young(01)$}
(-6,-8)node[]{$\tiny\young(0123)$}
(8,-10)node[]{$\tiny\young(0,\yminusone,\yminustwo,\yminusthree,\yminusfour)$}
(4,-10)node[]{$\tiny\young(0,\yminusone,\yminustwo)$}
(0,-10)node[]{$\tiny\young(0)$}
(-4,-10)node[]{$\tiny\young(012)$}
(-8,-10)node[]{$\tiny\young(01234)$}
(10,-12)node[]{$\tiny\young(0,\yminusone,\yminustwo,\yminusthree,\yminusfour,\yminusfive)$}
(6,-12)node[]{$\tiny\young(0,\yminusone,\yminustwo,\yminusthree)$}
(2,-12)node[]{$\tiny\young(0,\yminusone)$}
(-2,-12)node[]{$\tiny\young(01)$}
(-6,-12)node[]{$\tiny\young(0123)$}
(-10,-12)node[]{$\tiny\young(012345)$};
\draw[->,thick,magenta](-.5,-2.5)node{} to (-1.5,-3.5)node{};
\draw[->,thick,magenta](.5,-2.5)node{} to (1.5,-3.5)node{};
\draw[->,thick,magenta](-2.5,-4.5)node{} to (-3.5,-5.5)node{};
\draw[->,thick,magenta](-1.5,-4.5)node{} to (-.5,-5.5)node{};
\draw[->,thick,magenta](1.5,-4.5)node{} to (.5,-5.5)node{};
\draw[->,thick,magenta](2.5,-4.5)node{} to (3.5,-5.5)node{};
\draw[->,thick,magenta](-4.5,-6.5)node{} to (-5.5,-7.5)node{};
\draw[->,thick,magenta](-.5,-6.5)node{} to (-1.5,-7.5)node{};
\draw[->,thick,magenta](-3.5,-6.5)node{} to (-2.5,-7.5)node{};
\draw[->,thick,magenta](.5,-6.5)node{} to (1.5,-7.5)node{};
\draw[->,thick,magenta](3.5,-6.5)node{} to (2.5,-7.5)node{};
\draw[->,thick,magenta](4.5,-6.5)node{} to (5.5,-7.5)node{};
\draw[->,thick,magenta](-6.5,-8.5)node{} to (-7.5,-9.5)node{};
\draw[->,thick,magenta](5.5,-8.5)node{} to (4.5,-9.5)node{};
\draw[->,thick,magenta](1.5,-8.5)node{} to (.5,-9.5)node{};
\draw[->,thick,magenta](-2.5,-8.5)node{} to (-3.5,-9.5)node{};
\draw[->,thick,magenta](-5.5,-8.5)node{} to (-4.5,-9.5)node{};
\draw[->,thick,magenta](-1.5,-8.5)node{} to (-.5,-9.5)node{};
\draw[->,thick,magenta](2.5,-8.5)node{} to (3.5,-9.5)node{};
\draw[->,thick,magenta](6.5,-8.5)node{} to (7.5,-9.5)node{};
\draw[->,thick,magenta](-8.5,-10.5)node{} to (-9.5,-11.5)node{};
\draw[->,thick,magenta](-4.5,-10.5)node{} to (-5.5,-11.5)node{};
\draw[->,thick,magenta](-.5,-10.5)node{} to (-1.5,-11.5)node{};
\draw[->,thick,magenta](3.5,-10.5)node{} to (2.5,-11.5)node{};
\draw[->,thick,magenta](7.5,-10.5)node{} to (6.5,-11.5)node{};
\draw[->,thick,magenta](8.5,-10.5)node{} to (9.5,-11.5)node{};
\draw[->,thick,magenta](4.5,-10.5)node{} to (5.5,-11.5)node{};
\draw[->,thick,magenta](.5,-10.5)node{} to (1.5,-11.5)node{};
\draw[->,thick,magenta](-3.5,-10.5)node{} to (-2.5,-11.5)node{};
\draw[->,thick,magenta](-7.5,-10.5)node{} to (-6.5,-11.5)node{};
;}
$$
Reading across the row at level $n$ from left to right, the partitions $(n),(n-2),$ $\dots,(1^{n-2}),(1^n)$ correspond to $\CC_0(\Lambda^0 V_n),\CC_0(\Lambda^1 V_n),\dots,\CC_0(\Lambda^{n-2} V_n),\CC_0(\Lambda^{n-1} V_n).$

 Instead of labeling the vertices of the branching graph with partitions, let's label the vertices with $\dim\CC_0(\Lambda^k V_n)$ using restriction rule for the arrows. Doing this, we obtain Pascal's triangle:

$$\TikZ{[scale=.5]
\draw
(0,-2)node[]{$1$}
(2,-4)node[]{$1$}
(-2,-4)node[]{$1$}
(4,-6)node[]{$1$}
(0,-6)node[]{$2$}
(-4,-6)node[]{$1$}
(6,-8)node[]{$1$}
(2,-8)node[]{$3$}
(-2,-8)node[]{$3$}
(-6,-8)node[]{$1$}
(8,-10)node[]{$1$}
(4,-10)node[]{$4$}
(0,-10)node[]{$6$}
(-4,-10)node[]{$4$}
(-8,-10)node[]{$1$}
(10,-12)node[]{$1$}
(6,-12)node[]{$5$}
(2,-12)node[]{$10$}
(-2,-12)node[]{$10$}
(-6,-12)node[]{$5$}
(-10,-12)node[]{$1$};
\draw[->,thick,magenta](-.5,-2.5)node{} to (-1.5,-3.5)node{};
\draw[->,thick,magenta](.5,-2.5)node{} to (1.5,-3.5)node{};
\draw[->,thick,magenta](-2.5,-4.5)node{} to (-3.5,-5.5)node{};
\draw[->,thick,magenta](-1.5,-4.5)node{} to (-.5,-5.5)node{};
\draw[->,thick,magenta](1.5,-4.5)node{} to (.5,-5.5)node{};
\draw[->,thick,magenta](2.5,-4.5)node{} to (3.5,-5.5)node{};
\draw[->,thick,magenta](-4.5,-6.5)node{} to (-5.5,-7.5)node{};
\draw[->,thick,magenta](-.5,-6.5)node{} to (-1.5,-7.5)node{};
\draw[->,thick,magenta](-3.5,-6.5)node{} to (-2.5,-7.5)node{};
\draw[->,thick,magenta](.5,-6.5)node{} to (1.5,-7.5)node{};
\draw[->,thick,magenta](3.5,-6.5)node{} to (2.5,-7.5)node{};
\draw[->,thick,magenta](4.5,-6.5)node{} to (5.5,-7.5)node{};
\draw[->,thick,magenta](-6.5,-8.5)node{} to (-7.5,-9.5)node{};
\draw[->,thick,magenta](5.5,-8.5)node{} to (4.5,-9.5)node{};
\draw[->,thick,magenta](1.5,-8.5)node{} to (.5,-9.5)node{};
\draw[->,thick,magenta](-2.5,-8.5)node{} to (-3.5,-9.5)node{};
\draw[->,thick,magenta](-5.5,-8.5)node{} to (-4.5,-9.5)node{};
\draw[->,thick,magenta](-1.5,-8.5)node{} to (-.5,-9.5)node{};
\draw[->,thick,magenta](2.5,-8.5)node{} to (3.5,-9.5)node{};
\draw[->,thick,magenta](6.5,-8.5)node{} to (7.5,-9.5)node{};
\draw[->,thick,magenta](-8.5,-10.5)node{} to (-9.5,-11.5)node{};
\draw[->,thick,magenta](-4.5,-10.5)node{} to (-5.5,-11.5)node{};
\draw[->,thick,magenta](-.5,-10.5)node{} to (-1.5,-11.5)node{};
\draw[->,thick,magenta](3.5,-10.5)node{} to (2.5,-11.5)node{};
\draw[->,thick,magenta](7.5,-10.5)node{} to (6.5,-11.5)node{};
\draw[->,thick,magenta](8.5,-10.5)node{} to (9.5,-11.5)node{};
\draw[->,thick,magenta](4.5,-10.5)node{} to (5.5,-11.5)node{};
\draw[->,thick,magenta](.5,-10.5)node{} to (1.5,-11.5)node{};
\draw[->,thick,magenta](-3.5,-10.5)node{} to (-2.5,-11.5)node{};
\draw[->,thick,magenta](-7.5,-10.5)node{} to (-6.5,-11.5)node{};
;}
$$
Thus our Bratteli diagram of irreducible calibrated representations can be seen as a  categorification of Pascal's triangle. Pascal's triangle also shows up in the Bratteli diagrams of the blob algebra \cite{TheBlob} and the planar rook algebra \cite{PascalTriangle}.

Finally, there is one more natural way to label the vertices of our branching graph. Consider again the Young diagrams of partitions in $\cup_{n\geq 1}\mathcal{RC}_n$ labeling the vertices of $\Gamma$. There is an obvious combinatorial rule to build this graph of Young diagrams: start with the Young diagram consisting of a single box at level $1$ of the graph. The Young diagrams at level $n+1$ are built from those at level $n$ by adding or removing a single box so long as the result is a non-empty Young diagram with a single row or column. There is an arrow $\lambda\rightarrow \mu$ if $\lambda\in\mathcal{RC}_n$, $\mu\in\mathcal{RC}_{n+1}$, and $\mu$ is obtained from $\lambda$ by adding or removing a box. Now, we observe that a single row or column partition has a unique removable box. We then label the vertex occupied by that partition with the content of its removable box. This produces a graph whose vertices are integers, and which is generated by starting with $0$ and then has an arrow from $a$ at level $n$ to $b$ at level $n+1$ if and only if $b=a\pm 1$.

$$\TikZ{[scale=.5]
\draw
(0,-2)node[]{$0$}
(2,-4)node[]{$-1$}
(-2,-4)node[]{$1$}
(4,-6)node[]{$-2$}
(0,-6)node[]{$0$}
(-4,-6)node[]{$2$}
(6,-8)node[]{$-3$}
(2,-8)node[]{$-1$}
(-2,-8)node[]{$1$}
(-6,-8)node[]{$3$}
(8,-10)node[]{$-4$}
(4,-10)node[]{$-2$}
(0,-10)node[]{$0$}
(-4,-10)node[]{$2$}
(-8,-10)node[]{$4$}
(10,-12)node[]{$-5$}
(6,-12)node[]{$-3$}
(2,-12)node[]{$-1$}
(-2,-12)node[]{$1$}
(-6,-12)node[]{$3$}
(-10,-12)node[]{$5$};
\draw[->,thick,magenta](-.5,-2.5)node{} to (-1.5,-3.5)node{};
\draw[->,thick,magenta](.5,-2.5)node{} to (1.5,-3.5)node{};
\draw[->,thick,magenta](-2.5,-4.5)node{} to (-3.5,-5.5)node{};
\draw[->,thick,magenta](-1.5,-4.5)node{} to (-.5,-5.5)node{};
\draw[->,thick,magenta](1.5,-4.5)node{} to (.5,-5.5)node{};
\draw[->,thick,magenta](2.5,-4.5)node{} to (3.5,-5.5)node{};
\draw[->,thick,magenta](-4.5,-6.5)node{} to (-5.5,-7.5)node{};
\draw[->,thick,magenta](-.5,-6.5)node{} to (-1.5,-7.5)node{};
\draw[->,thick,magenta](-3.5,-6.5)node{} to (-2.5,-7.5)node{};
\draw[->,thick,magenta](.5,-6.5)node{} to (1.5,-7.5)node{};
\draw[->,thick,magenta](3.5,-6.5)node{} to (2.5,-7.5)node{};
\draw[->,thick,magenta](4.5,-6.5)node{} to (5.5,-7.5)node{};
\draw[->,thick,magenta](-6.5,-8.5)node{} to (-7.5,-9.5)node{};
\draw[->,thick,magenta](5.5,-8.5)node{} to (4.5,-9.5)node{};
\draw[->,thick,magenta](1.5,-8.5)node{} to (.5,-9.5)node{};
\draw[->,thick,magenta](-2.5,-8.5)node{} to (-3.5,-9.5)node{};
\draw[->,thick,magenta](-5.5,-8.5)node{} to (-4.5,-9.5)node{};
\draw[->,thick,magenta](-1.5,-8.5)node{} to (-.5,-9.5)node{};
\draw[->,thick,magenta](2.5,-8.5)node{} to (3.5,-9.5)node{};
\draw[->,thick,magenta](6.5,-8.5)node{} to (7.5,-9.5)node{};
\draw[->,thick,magenta](-8.5,-10.5)node{} to (-9.5,-11.5)node{};
\draw[->,thick,magenta](-4.5,-10.5)node{} to (-5.5,-11.5)node{};
\draw[->,thick,magenta](-.5,-10.5)node{} to (-1.5,-11.5)node{};
\draw[->,thick,magenta](3.5,-10.5)node{} to (2.5,-11.5)node{};
\draw[->,thick,magenta](7.5,-10.5)node{} to (6.5,-11.5)node{};
\draw[->,thick,magenta](8.5,-10.5)node{} to (9.5,-11.5)node{};
\draw[->,thick,magenta](4.5,-10.5)node{} to (5.5,-11.5)node{};
\draw[->,thick,magenta](.5,-10.5)node{} to (1.5,-11.5)node{};
\draw[->,thick,magenta](-3.5,-10.5)node{} to (-2.5,-11.5)node{};
\draw[->,thick,magenta](-7.5,-10.5)node{} to (-6.5,-11.5)node{};
;}
$$
This third presentation of the branching graph records the eigenvalues of the Jucys--Murphy elements acting on the basis of $Y$-eigenvectors for the vertices $\CC_0(\Lambda^k V_n)$.

Fix a level $n$ of the branching graph $\Gamma$, say $n=6$. Fix $\lambda$ at level $n$ of the graph; sticking with $n=6$, take, say, $\lambda={\tiny\yng(2)}$. Consider a path $T$ in the graph from the source vertex ${\tiny\yng(1)}$ at level $1$ to $\lambda$. The set of all such paths $T$ label the basis of eigenvectors for $\CC_0(\Lambda^k V_n)$ under the correspondence above. For example, we can take $T$ to be the path $(0,-1,0,1,0,1)$ from the source of the graph to ${\tiny\yng(2)}$ at level $6$. If $T'$ is another path from the source to the same vertex ${\tiny\yng(2)}$, we can get $T'$ from $T$ by modifying $T$ some number of times by ``going the other way around a diamond." So for example, if $T'=(0,1,0,1,0,1)$ then we obtain $T'$ from $T$ by going around the left side of the topmost diamond in the graph through ${\tiny\yng(2)}$ instead of around the right side of the diamond through ${\tiny\yng(1,1)}$. On our third, numerical graph, this corresponds to adding $2$ to the eigenvalue of $Y_2$. 

In general, let $\lambda$ be a vertex in $\Gamma$. The paths $T$ from the source vertex to $\lambda$ at level $n$, and thus the spectrum of the irreducible representation $\CC_0(\Lambda^k V_n)\in\sBr_n$-mod, are given by all $n$-tuples of integers of the following form: $$\Spec\CC_0(\Lambda^k V_n)=\{(a_1,\dots,a_n)\in\Z^n\mid a_1=0,\; a_n=\mathrm{ct}(\mathrm{b}_\lambda),\;a_{i+1}=a_i\pm 1 \hbox{ if }1\leq i\leq n-1\}$$
where $\ct(\mathrm{b}_\lambda)$ is the content of the unique removable box of $\lambda$, so 
$$\ct(\mathrm{b}_\lambda)=
\begin{cases} 
\:\:\:k-1&\hbox{ if }\lambda=k,\\
-k+1&\hbox{ if }\lambda=(1^k).
\end{cases}$$ Starting from any given path from ${\tiny\yng(1)}$ at level $1$ to $\lambda$ at level $n$, we can obtain any other path with the same starting and ending points by successively adding or subtracting $2$ from various $a_j$, so long as we do not change $a_1$, we do not change $a_n$, and at each step of the process we preserve the property that $a_{i+1}=a_i\pm 1$ for  all $1\leq i\leq n-1$. A canonical way of doing this is to start with the vector $\overline{v}=(0,1,2,\dots,n-2-k,n-1-k,n-2-k,\dots,n-2k,n-1-2k)$ which is the path to $\lambda$ whose first $n-k-1$ steps are to the left and the remaining $k$ steps are to the right. Then generate the rest of the paths by subtracting multiples of $2$ from entries of $\overline{v}$ in all possible ways while staying in $\Spec\LL(\lambda)$. In this way we can easily generate $\Spec\CC_0(\Lambda^k V_n)$ for any $\lambda$ a vertex in $\Gamma$, that is, we can algorithmically generate the matrices $y_1,\dots,y_n$ which act diagonally on $\CC_0(\Lambda^k V_n)$.

\begin{example}
Let $n=12$ and let $\lambda=(10)$. We calculate $\Spec \CC_0(\Lambda^1 V_{12})$: 
\begin{align*}
\Spec \CC_0(\Lambda^1 V_{12})=\{&(0,1,2,3,4,5,6,7,8,9,10,9),\\&(0,1,2,3,4,5,6,7,8,9,8,9),\\&(0,1,2,3,4,5,6,7,8,7,8,9),\\&(0,1,2,3,4,5,6,7,6,7,8,9),\\&(0,1,2,3,4,5,6,5,6,7,8,9),\\&
(0,1,2,3,4,5,4,5,6,7,8,9),\\&(0,1,2,3,4,3,4,5,6,7,8,9),\\&(0,1,2,3,2,3,4,5,6,7,8,9),\\&
(0,1,2,1,2,3,4,5,6,7,8,9),\\&(0,1,0,1,2,3,4,5,6,7,8,9),\\&(0,-1,0,1,2,3,4,5,6,7,8,9)
\}. 
\end{align*}
We can then read off the matrix for $Y_j$, $j=1,\dots,n$, as the $j$'th column of this array:  $Y_1=\diag(0,0,\dots,0)$, $Y_2=\diag(1,\dots,1,-1)$, $Y_3=\diag(2,\dots,2,0,0)$, and so on. In particular we see that $Y_j$ acting on $\CC_0(\Lambda^1 V_{n})$ has a simple closed form: it is given by the diagonal matrix whose first $n-j$ entries are $j-1$ and whose last $j-1$ entries are $j-3$. Symmetrically, there is an easy formula for the matrices of $Y_j$ acting on $\CC_0(\Lambda^{n-2} V_{n})$, and we leave this for the reader.
\end{example}

\subsection{Final remarks} 1. We \textit{expect} that the irreducible module $\CC_0(\Lambda^k V_n)$ is identified in the labeling conventions of \cite{Coulembier} (and possibly up to taking the transpose)
 as follows:
$$\CC_0(\Lambda^k V_n)=
\begin{cases} \LL(n-2k)\quad &\hbox{ if } k\leq \frac{n-1}{2},\\ 
\LL(1^{2(k+1)-n})\quad &\hbox { if } k\geq \frac{n-1}{2}. 
\end{cases}$$
That is, we should have:
\begin{table}[H]
\centering
\begin{tabular}{c|c|c|c|c|c|c}
$\CC_0(\Lambda^0 V_n)$ & $\CC_0(\Lambda^1 V_n)$ & $\CC_0(\Lambda^2 V_n)$ & $\cdots$ & $\CC_0(\Lambda^{n-3} V_n)$ & $\CC_0(\Lambda^{n-2} V_n)$ & $\CC_0(\Lambda^{n-1} V_n)$\\
\hline
&&&&&&\\
$\LL(n)$ & $\LL(n-2)$ & $\LL(n-4)$ & $\cdots$ & $\LL(1^{n-4})$ & $\LL(1^{n-2})$ & $\LL(1^{n})$
\end{tabular}
\end{table}
\noindent However, at the time of writing we don't understand how to make this identification rigorous. Our construction is self-contained and doesn't use cell modules, whereas the irreducible representation $\LL(\lambda)$, $\lambda\neq \emptyset$ a partition of $n$ or $n-2$ or $n-4$ or $\ldots$ is defined as the quotient of the cell module $\WW(\lambda)$ by its radical. The rule for eigenvalues in our Bratteli diagram does not completely match the rule for the eigenvalues  of the Jucys--Murphy elements on the cell modules in the Bratteli diagram of up-down tableaux: it differs in the case where a box is removed and the content of the removed box is positive \cite[Lemma 6.2.5]{Coulembier}. In \cite[Lemma 6.2.5]{Coulembier} the eigenvalue is always $\ct(b)+1$ where $b$ is the box removed. But in our branching graph of irreducible calibrated representations, the eigenvalue produced by removing a box is $\ct(b)+1$ if $\ct(b)<0$ and $\ct(b)-1$ if $\ct(b)>0$. It is worth remarking that our sign conventions for the defining relations of $\sBr_n$ that mix $e$'s and $s$'s are opposite to those of \cite{Moon} and \cite{Coulembier}. This has the consequence that when $n=2$ for example, the cell module $\WW(\emptyset)$ would have content sequence $(0,-1)$ in our convention, and $(0,1)$ in the conventions of \cite{Coulembier}. We expect that once the differing sign conventions are taken into account, the rest of the discrepancy arises because the cell module $\WW(\emptyset)\in\sBr_n$-mod, $n$ even, is actually irreducible and isomorphic to $\Lambda^{\frac{n}{2}} V_n$. The cell module $\WW(\emptyset)$ is the unique cell module equal to its radical and $\emptyset$ does not \textit{label} an irreducible representation. Perhaps some of the paths we have found factor through $\emptyset$ at even levels and this accounts for the discrepancy between the rule for cell modules in general and the rule for this particularly special tower of irreducible representations. 

2. Which irreducible $\sBr_n$-representations are calibrated? Between the representations $\CC_0(\Lambda^k V_n)$ constructed in this paper and the representations $\LL(\lambda)$ where  $\lambda$ is a partition of $n$ (the embedding of $\K S_n$-mod in $\sBr_n$-mod), have we found all of them or are there more? To frame the question in terms of $\Spec \LL$ where $\LL$ is an irreducible calibrated representation: in this paper we constructed all irreducible $\LL\in\sBr_n$-mod such that the $n$-tuples of eigenvalues $(\alpha_1,\alpha_2,\dots,\alpha_n)\in\Spec\LL$ satisfy $\alpha_j-\alpha_{j+1}=\pm 1$ for all $j=1,\dots,n-1$. Are there irreducible calibrated representations $\LL\in\sBr_n$-mod for which $\alpha_j-\alpha_{j+1}$ is not always $\pm 1$ but on which $e_i$ doesn't act by $0$?\\

\section*{Appendix}\label{code}
\subsection{Code for Theorem~\ref{reflection rep}} We check that the matrices given in Theorem \ref{reflection rep} satisfy the defining relations of $\sv_n$ in Definition \ref{def svn} using the
code below.

\tiny 
\begin{lstlisting}[language=Mathematica,caption={Code for Theorem~\ref{reflection rep}}]
n=7; (*number of strands*)
m=n-1; (*size of matrix representation*)
a; (*eigenvalue of y[1]*)
y[1]=a IdentityMatrix[m];
e[1]=DiagonalMatrix[{1},1,m];
s[1]=IdentityMatrix[m]+DiagonalMatrix[{-2},0,m]+DiagonalMatrix[{-1},1,m];
e[m]=Reverse/@(Transpose[Reverse/@DiagonalMatrix[{-1},-1,m]]);
s[m]=Reverse/@(Transpose[Reverse/@DiagonalMatrix[{-1},-1,m]])+Reverse/@(
    Transpose[Reverse/@DiagonalMatrix[{-2},0,m]])+IdentityMatrix[m];
For[i=2,i<m,i++,{e[i]=DiagonalMatrix[Table[-KroneckerDelta[i-1,k],{k,i-1}],-1,m]+DiagonalMatrix[Table[KroneckerDelta[i,k],{k,i}],1,m],s[i]=DiagonalMatrix[Table[-2 KroneckerDelta[i,k],{k,i}],0,m]+DiagonalMatrix[-Table[KroneckerDelta[i-1,k],{k,i-1}],-1,m]+DiagonalMatrix[-Table[KroneckerDelta[i,k],{k,i}],1,m]+IdentityMatrix[m]}]; 
For[i=2,i<=n,i++,y[i]=y[i-1]+DiagonalMatrix[-2 
    Table[KroneckerDelta[i-1,k],{k,i-1}],0,m]+IdentityMatrix[m]];
For[i=1,i<n,i++,Print["sVW1,i=", i,",",MatrixForm[s[i].s[i]-IdentityMatrix[m]]]]
For[j=1,j<n,j++,For[i=1,Abs[i-j]>1,i++,Print["sVW2(i),{j,i}=",{j,i},",",
    MatrixForm[s[i].e[j]-e[j].s[i]]]]]
For[j=1,j<n,j++,For[i=1,Abs[i-j]>1,i++,Print["sVW2(i),{j,i}=",{j,i},",",
    MatrixForm[s[j].e[i]-e[i].s[j]]]]]
For[j=1,j<n,j++,For[i=1,Abs[i-j]>1,i++,Print["sVW2(ii),{j,i}=",{j,i},",",
    MatrixForm[e[i].e[j]-e[j].e[i]]]]]
For[j=1,j<=n,j++,For[i=1,Abs[i-j]>1,i++,Print["sVW2(iii),{j,i}=",{j,i},",",
    MatrixForm[e[i].y[j]-y[j].e[i]]]]] 
For[i=1,i<n,i++,For[j=1,Abs[i-j]>1,j++,Print["sVW2(iii),{j,i}=",{j,i},",",
    MatrixForm[e[i].y[j]-y[j].e[i]]]]] 
For[i=2,i<=n-1,i++,Print["sVW2(iii),{j=i-1,i}=",{i-1,i},",",
    MatrixForm[e[i].y[i-1]-y[i-1].e[i]]]]
For[i=1,i<=n,i++,For[j=1,j<=n,j++, Print["sVW2(iv),i=",i,",",
    MatrixForm[y[i].y[j]-y[j].y[i]]]]]
For[i=1,i<n,i++,For[j=1,Abs[i-j]>1,j++,Print["sVW3(i),{j,i}=",{j,i},",",
    MatrixForm[s[i].s[j]-s[j].s[i]]]]]
For[i=1,i<n-1,i++,Print["sVW3(ii),i=",i,",",
    MatrixForm[s[i].s[i+1].s[i]-s[i+1].s[i].s[i+1]]]]
For[j=1,j<=n,j++,For[i=1,Abs[i-j]>1,i++,
    Print["sVW3(iii),{j,i}=",{j,i},",",MatrixForm[s[i].y[j]-y[j].s[i]]]]]
For[i=1,i<n,i++,For[j=1,Abs[i-j]>1,j++,
    Print["sVW3(iii),{j,i}=",{j, i},",",MatrixForm[s[i].y[j]-y[j].s[i]]]]]
For[i=2,i<=n-1,i++,Print["sVW3(iii),{j=i-1,i}=",{i-1,i},",",MatrixForm[s[i].y[i-1]-y[i-1].s[i]]]]
For[i=1,i<n-1,i++,Print["sVW4(i),i=",i,",",MatrixForm[e[i+1].e[i].e[i+1]+e[i+1]]]]
For[i=1,i<n-1,i++,Print["sVW4(ii),i="i,",",MatrixForm[e[i].e[i+1].e[i]+e[i]]]]
For[i=1,i<n,i++,Print["sVW5(i)a,i=",i,",",MatrixForm[e[i].s[i]-e[i]],",sVW5(i)b,i=",i,",",
    MatrixForm[s[i].e[i]+e[i]]]]
For[i=1,i<n-1,i++,Print["sVW5(ii),i=",i,",",MatrixForm[s[i].e[i+1].e[i]-s[i+1].e[i]]]]
For[i=1,i<n-1,i++,Print["sVW5(iii),i=",i,",",MatrixForm[s[i+1].e[i].e[i+1]+s[i].e[i+1]]]]
For[i=1,i<n-1,i++,Print["sVW5(iv),i=",i,",",MatrixForm[e[i+1].e[i].s[i+1]-e[i+1].s[i]]]]
For[i=1,i<n-1,i++,Print["sVW5(v),i=",i,",",MatrixForm[e[i].e[i+1].s[i]+e[i].s[i+1]]]]
For[i=1,i<n,i++,Print["sVW6,i=",i,",",MatrixForm[e[i].e[i]]]]
For[i=1,i<n,i++,Print["sVW7(i),i=",i,",",MatrixForm[s[i].y[i]-y[i+1].s[i]+e[i]+IdentityMatrix[m]]]]
For[i=1,i<n,i++,Print["sVW7(ii),i=",i,",",MatrixForm[y[i].s[i]-s[i].y[i+1]-e[i]+IdentityMatrix[m]]]]
For[i=1,i<n,i++,Print["sVW8(i),i=",i,",",MatrixForm[e[i].(y[i]-y[i+1])+e[i]]]]
For[i=1,i<n,i++,Print["sVW8(ii),i=",i,",",MatrixForm[(y[i]-y[i+1]).e[i]-e[i]]]]
\end{lstlisting}
\normalsize

\section*{Acknowledgments}{\footnotesize
This material is based upon work supported by the National Science Foundation under Grant No. DMS-1440140, and the National Security Agency under Grant No. h98230-18-1-0144 while the authors were in residence at the Mathematical Sciences Research Institute in Berkeley, California, for two weeks during the Summer of 2018. Lemma \ref{J rad sv2} was proved during that residence by the first author after conversations with our co-residents in the program, Z. Daugherty and I. Halacheva. The problems of determining the Jacobson radical and calibrated representations of $\sv_n$ were discussed there with Z. Daugherty and I. Halacheva, and several of their computations of matrix equations arising from a calibrated representation of $\sv_2$ led to our first paper joint with them \cite{some of us} and are used here in Lemma \ref{basicbitch}. We are grateful to Chris Bowman and Kevin Coulembier for helpful communications.
}


\begin{thebibliography}{10}

\bibitem{us+}
M. Balagovic, Z. Daugherty, I. Entova-Aizenbud, I. Halacheva, J. Hennig, M.S. Im, G. Letzer, E. Norton, V. Serganova, and C. Stroppel. The affine VW supercategory, arXiv:1801.04178.

\bibitem{Bowmanetal} 
C. Bowman, M. De Visscher, and J. Enyang.
Simple modules for the partition algebra and monotone convergence of Kronecker coefficients. 
{\em Int. Math. Res. Not.} IMRN 2019, no. 4, 1059--1097. 

\bibitem{BowmanEnyangGoodman}
C. Bowman, J. Enyang, and F. Goodman.
Diagram algebras, dominance triangularity and skew cell modules. 
{\em J. Aust. Math. Soc.} 104 (2018), no. 1, 13--36. 

\bibitem{BNS}
C. Bowman, E. Norton, and J. Simental.
Characteristic-free bases and BGG resolutions of unitary simple modules for quiver Hecke and Cherednik algebras. 
arXiv:1803.08736.

\bibitem{ChenPeng}
C.-W. Chen and Y.-N. Peng.
Affine periplectic Brauer algebras. 
{\em J. Algebra} 501 (2018), 345--372. 

\bibitem{Cherednik}
I. Cherednik. On special bases of irreducible finite-dimensional representations of the
degenerate affine Hecke algebra. {\em Func. Anal. Appl.} 20 (1986), 76–78.

\bibitem{Coulembier}
K. Coulembier.
The periplectic Brauer algebra. 
 {\em Proc. Lond. Math. Soc.} (3) 117 (2018), no. 3, 441--482. 
 
\bibitem{CoulembierEhrig1}
K. Coulembier and M. Ehrig.
The periplectic Brauer algebra II: Decomposition multiplicities. 
{\em J. Comb. Algebra} 2 (2018), no. 1, 19--46. 

\bibitem{CoulembierEhrig2}
K. Coulembier and M. Ehrig.
The periplectic Brauer algebra III: The Deligne category.
arXiv:1704.07547

\bibitem{ChrissGinzburg}
N. Chriss and V. Ginzburg.
{ \em Representation Theory and Complex Geometry}.
Birkh\"{a}user Boston, 1997.

\bibitem{some of us}
Z. Daugherty, I. Halacheva, M.S. Im, and E. Norton. On calibrated representations of the degenerate affine periplectic Brauer algebra, arXiv:1905.05148.

\bibitem{Drinfeld}
V. G. Drinfeld. Degenerate affine Hecke algebras and Yangians, {\em Funktsional. Anal. i Prilozhen.}, 20:1 (1986), 69--70; {\em Funct. Anal. Appl.}, 20:1 (1986), 58--60.


\bibitem{EncyclMath}
``Jacobson radical."
The Encyclopaedia of Mathematics, vol. 5. Ed. Vinogradov. Kluwer Academic Publishers, 1990.

\bibitem{Enyang1}
J. Enyang.
Jucys-Murphy elements and a presentation for partition algebras. 
{\em J. Algebraic Combin.} 37 (2013), no. 3, 401--454. 

\bibitem{Enyang2}
J. Enyang.
A seminormal form for partition algebras. 
{\em J. Combin. Theory Ser. A} 120 (2013), no. 7, 1737--1785. 


\bibitem{PascalTriangle}
D. Flath, T. Halverson, and K. Herbig.
The planar rook algebra and Pascal's triangle.
{\em Enseign. Math.} (2) 55 (2009), no. 1-2, 77--92. 

\bibitem{FultonHarris}
W. Fulton and J. Harris.
{\em Representation theory. A first course.}  Graduate Texts in Mathematics, 129. Readings in Mathematics. Springer-Verlag, New York, 1991.

\bibitem{TheBlob}
A. Gainutdinov, J.L. Jacobsen, H. Saleur, and R. Vasseur.
A physical approach to the classification of indecomposable Virasoro representations from the blob algebra.
{\em Nuclear Phys. B} 873 (2013), no. 3, 614--681. 

\bibitem{Jucys}
A. Jucys. Symmetric polynomials and the center of the symmetric group ring. {\em Reports
Math. Phys.} 5 (1974), 107--112.

\bibitem{Kleshchev}
A. Kleshchev. {\em Linear and Projective Representations of Symmetric Groups.} 
Cambridge Tracts in Mathematics 163, Cambridge University Press, 2005.

\bibitem{KriloffRam}
C. Kriloff and A. Ram.
Representations of graded Hecke algebras. 
{\em Represent. Theory} 6 (2002), 31--69. 

\bibitem{KujawaTharp}
J.R. Kujawa and B.C. Tharp.
The marked Brauer category. 
{\em J. Lond. Math. Soc.} (2) 95 (2017), no. 2, 393--413. 

\bibitem{Lusztig} G. Lusztig, Affine Hecke algebras and their graded version, {\em J. Am. Math. 
Soc.} 2, No 3 (1989), 599--635. 


\bibitem{Moon}
D. Moon.
The centralizer algebra of the Lie superalgebra $\mathfrak{p}(n)$ and the decomposition of $V^{\otimes k}$ as a $\mathfrak{p}(n)$-module. {\em Recent progress in algebra} (Taejon/Seoul, 1997), 199--212,
{\em Contemp. Math.}, 224, Amer. Math. Soc., Providence, RI, 1999. 

\bibitem{Murphy}
G. Murphy. A new construction of Young’s seminormal representation of the symmetric
group. {\em J. Algebra} 69 (1981), 287--291.

\bibitem{Nazarov}
M. Nazarov.
Young's orthogonal form for Brauer's centralizer algebra.
{\em J. Algebra} 182 (1996), no. 3, 664--693. 

\bibitem{OkounkovVershik}
A. Okounkov and A. Vershik.
A new approach to representation theory of symmetric groups.
{\em Selecta Math. (N.S.)} 2 (1996), no. 4, 581--605. 


\bibitem{Ruff} 
O. Ruff.
Completely splittable representations of symmetric groups and affine Hecke algebras.
{\em J. Algebra} 305 (2006), no. 2, 1197--1211. 



\end{thebibliography}
\end{document}